\documentclass[reqno,11pt]{amsart}
\usepackage{citesort}
\usepackage{enumerate}

\newtheorem{theorem}{Theorem}[section]
\newtheorem{proposition}[theorem]{Proposition}
\newtheorem{lemma}[theorem]{Lemma}

\theoremstyle{definition}

\theoremstyle{remark}
\newtheorem{remark}[theorem]{Remark}

\numberwithin{equation}{section}

\def\R{{\mathbb{R}}}

\def\ff{{II}}
\def\epsilon{{\varepsilon}}
\def\phi{{\varphi}}
\def\theta{{\vartheta}}
\def\dt{{\frac{d}{dt}}}

\DeclareMathOperator{\graph}{graph}

\def\ol#1{{\overline{#1}}}
\def\ul#1{{\underline{#1}}}

\begin{document}
\title{Entire scalar curvature flow and hypersurfaces of constant scalar curvature in Minkowski space}

\author{Pierre Bayard}
\thanks{The author was supported by the project DGAPA-UNAM IN 101507.}
\address{Pierre Bayard: Instituto de F\'{\i}sica y Matem\'aticas. U.M.S.N.H. Ciudad Universitaria. CP. 58040 Morelia, Michoac\'an, Mexico.}
\email{bayard@ifm.umich.mx}

\begin{abstract}
We prove existence in the Minkowski space of entire spacelike hypersurfaces with constant negative scalar curvature and given set of lightlike directions at infinity;  we also construct the entire scalar curvature flow with prescribed set of lightlike directions at infinity, and prove that the flow converges to a spacelike hypersurface with constant scalar curvature. The proofs rely on barriers construction and a priori estimates.
\end{abstract}
\maketitle
\markboth{PIERRE BAYARD}{ENTIRE SCALAR CURVATURE FLOW IN MINKOWSKI SPACE}

\section{Introduction}

The Minkowski space $\R^{n,1}$ is the space $\R^{n}\times\R$ endowed with the metric $dx_1^2+\cdots+dx_n^2-dx_{n+1}^2.$ We say that a hypersurface of $\R^{n,1}$ is spacelike if the metric induced on it by the Minkowski metric is Riemannian, and that a function $u:\R^n\rightarrow\R$ of class $C^1$ is spacelike if its graph is a spacelike hypersurface, which equivalently means that $|Du|<1$ on $\R^n.$ The principal curvatures of a spacelike hypersurface are the eigenvalues of its curvature endomorphism $dN,$ where $N$ is the future oriented unit normal field.  In the natural chart $(x_1,\ldots,x_n),$ the curvature endomorphism $\left(h^i_j\right)_{ij}$ of the graph of a spacelike function $u$ is given by
$$h^i_j=\frac{1}{\sqrt{1-|Du|^2}}\sum_{k=1}^n\left(\delta_{ik}+\frac{u_iu_k}{1-|Du|^2}\right)u_{kj}.$$
Let us denote by $H_k[u]$ the $k^{th}$ elementary symmetric function of  the principal curvatures of the graph of $u.$

We are interested in the scalar curvature $S[u]$ of the graph of $u,$ which is linked to $H_2[u]$ by
$$S[u]=-2H_2[u].$$
We say that $u:\R^n\rightarrow\R$ of class $C^2$ is admissible, if $u$ is spacelike and if $H_1[u]>0$ and $H_2[u]>0$ on $\R^n.$ It is well known that the operator $H_2$ is elliptic on admissible functions, and that the Mac-Laurin inequality holds: on $\R^n,$
\begin{equation}\label{mac laurin}
H_2[u]^{\frac{1}{2}}\leq\sqrt{\frac{n-1}{2n}}H_1[u].
\end{equation}
Let $F$ be a closed subset of the unit sphere $S^{n-1}\subset\R^n.$ We suppose that $F$ is a union of arcs of circles on $S^{n-1}.$ We first construct barriers whose set of lightlike directions at infinity is the set $F.$ For definitions and examples, we refer to Sections \ref{section lightlike directions} and \ref{section barriers}.
\begin{proposition}\label{proposition intro barriers}
Let $F$ be as above, and consider $V_F:\R^n\rightarrow\R$ defined by  $V_F(x):=\sup_{\lambda\in F}\langle x,\lambda\rangle,$ where $\langle.,.\rangle$ stands for the canonical scalar product on $\R^n.$ Let $h$ and $k$ be two positive constants such that
\begin{equation}\label{condition h k}
h<\sqrt{\frac{2n}{n-1}}\mbox{ and }k\geq 1.
\end{equation} 
There exist two entire functions
$\ul u,\ol u:\R^n\rightarrow\R,$ such that
\begin{equation}\label{inequality barriers}
V_F<\ul u<\ol u<V_F+c\mbox{ on }\R^n
\end{equation}
for some constant $c,$ where $\ol u$ is smooth, spacelike, with constant mean curvature $H_1=h,$ and $\ul u$ is the supremum of spacelike functions with constant scalar curvature $H_2=k.$ Moreover, for all $\xi\in F,$
\begin{equation}\label{limsup barrier intro}
\lim_{r\rightarrow+\infty}\ \ul u(r\xi)-r=\lim_{r\rightarrow+\infty}\ \ol u(r\xi)-r= 0,
\end{equation}
and, from (\ref{inequality barriers}), for all $\xi\in S^{n-1}\backslash F,$
\begin{equation}\label{limsup barrier intro2}
\lim_{r\rightarrow+\infty}\ \ol u(r\xi)-r= -\infty.
\end{equation}
\end{proposition}
Solving a sequence of Dirichlet problems between the barriers $\ul u$ and $\ol u,$ and extracting a convergent subsequence thanks to local estimates \cite{Bay,Bay2,U}, we will first construct an entire spacelike hypersurface of constant negative scalar curvature, and whose set of lightlike directions at infinity is $F$:
\begin{theorem}\label{entire elliptic problem} Let $F$ be a closed subset of $S^{n-1}$ as above. Then there exists $u:\R^n\rightarrow\R,$ admissible, solution of
\begin{equation}\label{entire elliptic equation}
H_2[u]=1\mbox{ in }\R^n
\end{equation}
such that, for all $\xi\in F,$
\begin{equation}\label{asymptotic u along F}
\lim_{r\rightarrow+\infty}u(r\xi)-r=0
\end{equation}
and
\begin{equation}\label{asymptotic u on Rn}
\sup_{\R^n}\left|u-V_F\right|<+\infty.
\end{equation}
In particular, the set of lightlike directions at infinity of $u$ is the set $F.$
\end{theorem}
\begin{remark}
Uniqueness of a solution of (\ref{entire elliptic equation}) satisfying (\ref{asymptotic u along F}) and (\ref{asymptotic u on Rn}) is still an open question.
\end{remark}
We then study the entire scalar curvature flow. Starting with a smooth spacelike entire and strictly convex function between the barriers which has bounded scalar curvature, we prove that the entire scalar curvature flow is defined for all time and converges to a solution of the prescribed constant scalar curvature equation:
\begin{theorem}\label{entire parabolic problem}
Let $F$ be as above. We suppose that $F$ is not included in any affine hyperplane of $\R^n.$ Let $h,k$ be two positive constants such that (\ref{condition h k}) holds, and let $\ul u,\ol u$ be the barriers given by Proposition \ref{proposition intro barriers}. Let $u_0:\R^n\rightarrow\R$ be a smooth spacelike and strictly convex function such that
\begin{equation}\label{u0 between barriers}
\ul u<u_0<\ol u
\end{equation}
and 
$$\displaystyle{1\leq H_2[u_0]\leq k.}$$ The parabolic problem 
\begin{equation}
\label{entire parabolic equation}
\left\{\begin{array}{rcl}
-\frac{\dot u}{\sqrt{1-|Du|^2}}+{H_2[u]}^{\frac{1}{2}}& = & 1\mbox{ in }\R^n\times (0,+\infty)\\
u(x,0)& = &u_0(x)\mbox{ on }\R^n\times\{0\},
\end{array}\right.
\end{equation}
has a smooth spacelike solution 
$$u\in C^{\infty}(\R^n\times(0,+\infty))\cap C^{1,1;0,1}(\R^n\times[0,+\infty)).$$
Moreover 
\begin{equation}\label{u between barriers1}
\ul u\leq u\leq \ol u
\end{equation}
for all time, and $u$ converges to a solution of (\ref{entire elliptic equation}) as the time $t$ tends to infinity.
\end{theorem}
\begin{remark} Note that (\ref{entire parabolic equation}) describes hypersurfaces moving with normal velocity given by the square root of the scalar curvature,
$$\dt X=\left(H_2[X]^{\frac{1}{2}}-1\right)N,$$
where $X$ is the embedding vector of the hypersurfaces. 
\end{remark}
\begin{remark}
If $F$ is included in some affine hyperplane, condition (\ref{u0 between barriers}) with $u_0$ strictly convex is not possible: suppose that $\xi\in\R^n$ belongs to $F^{\perp};$ then $V_F(\xi)=0,$ and, by (\ref{inequality barriers}) and (\ref{u0 between barriers}), $0<u_0(\lambda\xi)<c$ for all $\lambda\in\R,$ which is impossible if $\xi\neq 0$ and if $u_0$ is a strictly convex function. Note that the strictly convexity of $u_0$ is a crucial hypothesis for the resolution of the parabolic Dirichlet problem, Section \ref{section parabolic Dirichlet problem} . See also \cite{Bay,U}.
\end{remark}
\begin{remark}
If $u_0:\R^n\rightarrow\R$ is a spacelike and strictly convex function such that $1\leq H_2[u_0]\leq k$ and $\lim_{|x|\rightarrow+\infty}u_0(x)-|x|=0,$ we get the following: taking for the lower barrier $\ul u$ (resp. for the upper barrier $\ol u$) the hyperboloid asymptotic to the cone $x_{n+1}=|x|$ and of scalar curvature $H_2=k'>k$ (resp. of mean curvature $H_1=h<\sqrt{\frac{2n}{n-1}}$), by the maximum principle we have $\ul u< u_0< \ol u$, and Theorem \ref{entire parabolic problem} shows that problem (\ref{entire parabolic equation}) has a (unique) solution $u$ such that $\ul u\leq u\leq\ol u$ during the evolution. Moreover $u$ converges to the hyperboloid of scalar curvature $H_2=1,$ as $t$ tends to infinity.
\end{remark}
\begin{remark}
By scaling $u$ in Theorem \ref{entire elliptic problem}, we obtain an admissible solution of $H_2[u]=\lambda^2$ in $\R^n$ such that (\ref{asymptotic u along F}) and (\ref{asymptotic u on Rn}) hold. Moreover, by scaling the barriers $\ul u,\ol u$ in Proposition \ref{proposition intro barriers}, we obtain a result similar to Theorem \ref{entire parabolic problem} for the parabolic problem
\begin{equation}
\left\{\begin{array}{rcl}
-\frac{\dot u}{\sqrt{1-|Du|^2}}+{H_2[u]}^{\frac{1}{2}}& = & \lambda\mbox{ in }\R^n\times (0,+\infty)\\
u(x,0)& = &u_0(x)\mbox{ on }\R^n\times\{0\},
\end{array}\right.
\end{equation}
if $\ul u<u_0<\ol u$ and $\lambda^2\leq H_2[u_0]\leq\lambda^2k$ hold.
\end{remark}
Let us quote some related papers: in Minkowski space, entire spacelike hypersurfaces of constant mean curvature are classified in \cite{T} and entire hypersurfaces of constant Gauss curvature are studied in \cite{GJS,BaySch}. In \cite{Bay}, we construct entire hypersurfaces with prescribed scalar curvature and given values at infinity which stay at a bounded distance of a lightcone. 

The entire mean curvature flow in Minkowski space is studied in \cite{Ecker}, and the entire Gauss curvature flow in \cite{BaySch}. The scalar curvature flow in globally hyperbolic Lorentzian manifolds having a compact Cauchy hypersurface is studied in \cite{Ge1,Ge2} and \cite{E}. 

Finally, the parabolic Dirichlet problem for the scalar curvature operator in the euclidian space is solved in \cite{IL1,IL2}.
\\

The outline of the paper is as follows. We recall the definition of the set of lightlike directions at infinity of a spacelike and convex function in Section \ref{section lightlike directions}. In Section \ref{section barriers} we construct the barriers with given set of lightlike directions at infinity, and construct the auxiliary functions needed for the local estimates. The entire solutions of the prescribed constant scalar curvature equation are constructed Section \ref{section entire solution elliptic}. We introduce further notation and recall the evolution equations of various geometric quantities Section \ref{section evolution equations}, and we study the parabolic Dirichlet problem Section \ref{section parabolic Dirichlet problem}. In Section \ref{section entire solutions parabolic} we construct the entire scalar curvature flow, once local $C^1$ and $C^2$ estimates are known, and we prove that the flow converges. We carry out the local estimates in Sections \ref{section local C1 estimate} and \ref{section local C2 estimate}. A short appendix ends the paper.
\section{The set of lightlike directions at infinity of an entire spacelike hypersurface of constant scalar curvature}\label{section lightlike directions}
Let $u:\R^n\rightarrow\R$ be a spacelike and \textit{convex} function. Following Treibergs \cite{T}, its \textit{blow down} $V_u:\R^n\rightarrow\R$ is defined by
$$V_u(x)=\lim_{r\rightarrow+\infty}\frac{u(rx)}{r}.$$
As in \cite{T}, we denote by $Q$ the set of the convex homogeneous of degree one functions whose gradient has norm one whenever defined. The following holds:
\begin{lemma}
For every convex and spacelike solution $u$ of the prescribed scalar curvature equation (\ref{entire elliptic equation}), the blow down $V_u$ belongs to $Q.$ 
\end{lemma}
\begin{proof}
This result is proved in \cite{T}, Theorem 1 for the prescribed mean curvature equation, using a barrier construction. The same barrier can be used for the prescribed constant scalar curvature equation as well.
\end{proof}
The set $Q$ is in one-to-one correspondence with the set of closed subsets of $S^{n-1};$ see \cite{CT}, Lemma 4.3.
\begin{lemma}\cite{CT,T}.
If $F$ is a closed non-empty subset of $S^{n-1},$ 
$$V_F(x):=\sup_{\lambda\in F}\langle x,\lambda\rangle$$
belongs to $Q;$ the map $F\mapsto V_F$ is one-to-one, and its inverse is the map
$$w\in Q\mapsto F=\{x\in S^{n-1}\subset\R^n:\ w(x)=1\}.$$
\end{lemma}
In particular, the blow down of a convex solution $u$ of (\ref{entire elliptic equation}) is determined by the set of its \textit{lightlike directions} at infinity
$$L_u:=\{x\in S^{n-1}:\ V_u(x)=1\}.$$
Note that here, and in contrast with \cite{T,BaySch}, it is not known if a spacelike entire function of constant negative scalar curvature is necessarily convex. Nevertheless, the solutions $u$ constructed in this article are such that $V_F\leq u\leq V_F+c,$ where $F$ is a closed subset belonging to $S^{n-1}$ and $c$ is a constant. In that case, the blow down  $V_u$ and the set of lightlike directions $L_u$ are well-defined, and satisfy
$$V_u=V_F\mbox{ and }L_u=F.$$ 
We finally recall a useful formula. Denoting by $d_S$ the canonical distance on the sphere $S^{n-1},$ we proved the following formula in \cite{BaySch}, Lemma 4.6: for every $x\in S^{n-1},$
\begin{equation}\label{formula VF}
V_F(x)=\cos(d_S(x,F)).
\end{equation}
\section{The construction of the barriers}\label{section barriers}
\subsection{The semitrough}\label{section definition semitrough}
We first recall the properties of the \textit{standard semi\-trough} of constant Gauss curvature in the Minkowski space $\R^{2,1},$ constructed in \cite{HN}: this is the unique spacelike function $\tilde u:\R^2\rightarrow\R$ whose graph has constant Gauss curvature one, and which is such that  
$$D\tilde u(\R^2)=\{(x_1,x_2)\in B_1:\ x_1>0\},$$
and
\begin{equation}\label{asymptotic semitrough}
\lim_{|x|\rightarrow+\infty}\tilde u(x)-V_{S^+}(x)=0.
\end{equation}
Here $B_1$ is the unit ball in $\R^2$ centered at $0,$ and $S^+$ is the arc of the circle $S^1=\partial B_1$ defined by $S^+:=\{(x_1,x_2)\in S^1:\  x_1\geq 0\}.$ Let $S$ be a closed arc of circle on the sphere $S^{n-1}.$ This is a subset of the form $f(S^+\times \{0\}),$ where
\begin{equation}\label{definition S0}
S^+\times\{0\}=\{(x_1,x_2,0,\ldots,0)\in S^{n-1}:\ x_1\geq 0\}
\end{equation}
and $f$ is a conformal transformation of $S^{n-1}.$ From the existence of the standard semitrough, we deduce the following
\begin{lemma} \label{lemma semitrough segment} Let $S$ be a closed arc of circle on $S^{n-1},$ and let $k>0.$ There exists a spacelike entire function $u$ such that 
$$H_2[u]=k\mbox{ and }\sup_{\R^n}|u-V_S|<+\infty.$$
\end{lemma}
\begin{proof}
Recall that a Lorentz transformation preserves $H_2,$ and acts as a conformal transformation of $S^{n-1}$ on the sets of lightlike directions at infinity ($S^{n-1}$ is identified with the projective lightcone). Thus, applying a Lorentz transformation, we may suppose that $S$ is given by (\ref{definition S0}). The function $u$ defined by
$$u(x_1,x_2,x_3,\ldots,x_n)=\frac{1}{\sqrt{k}}\tilde u(\sqrt{k}(x_1,x_2)),$$
where $\tilde u$ is the standard semitrough defined above, satisfies the required properties.
\end{proof}
\subsection{The barriers}
Let $h$ be a positive constant and $F$ be a closed subset of $S^{n-1}.$ From \cite{T} p233, we know that there exists a smooth spacelike function $\ol u:\R^n\rightarrow\R$ whose graph has constant mean curvature $H_1=h$ and which is such that
\begin{equation}\label{bounds upper barrier}
V_F\leq \ol u\leq V_F+\frac{n}{h}\mbox{ on }\R^n.
\end{equation}
The function $\ol u$ satisfies the further properties: 
\begin{equation}\label{limsup upper barrier}
\limsup_{|x|\rightarrow+\infty}\ \ol u(x)-|x|\leq 0,
\end{equation}
and, for all $\xi\in F,$
\begin{equation}\label{asymptotic upper barrier along F}
\lim_{r\rightarrow+\infty}\ol u(r\xi)-r=0.
\end{equation}
For these last properties, see the upper barrier $z_2$ used in \cite{T} p233.
\begin{lemma} Let $h$ and $k$ be two positive constants such that $h<2\sqrt{k}.$ We assume that $F$ is a union of closed arcs of circles on $S^{n-1},$ 
\begin{equation}\label{form F}
F=\cup_{i\in I}S_i.
\end{equation}
Denoting by $u_i$ the entire spacelike function of constant scalar curvature $k$ associated to $S_i$ by Lemma \ref{lemma semitrough segment} (and its proof), the function $\displaystyle{\ul u=\sup_{i\in I}u_i}$ satisfies
$$V_F<\ul u< \ol u\mbox{ on }\R^n.$$
\end{lemma}
\begin{remark}\label{examples F}
The closure of an open subset $U$ of $S^{n-1}$ with $C^1$ boundary is of the form (\ref{form F}). More generally, if $U$ satisfies an interior cone condition at each boundary-point (i.e. all $\xi\in\partial U$ is a vertex of a (geodesic) cone $\subset\ol U$), $\ol U$ is of the form (\ref{form F}). Of course, the set $F$ in (\ref{form F}) might be much more complicated (e.g. without interior point). 
\end{remark}
\begin{proof}
We first prove that $V_F<\ul u$ on $\R^n.$ Since $u_i>V_{S_i}$ for all $i\in I,$ it is sufficient to prove that, for all $x\in\R^n,$ $V_F(x)=V_{S_i}(x)$ for some index $i\in I.$ Since these functions are homogeneous of degree one, we may suppose that $x\in S^{n-1}\subset\R^n.$ By (\ref{formula VF}), this amounts to prove that $d_S(x,F)=d_S(x,S_i)$ for some index $i\in I,$ where $d_S$ is the natural distance on $S^{n-1}.$ Let $x_0\in F$ be such that $d_S(x,F)=d_S(x,x_0),$ and $i\in I$ be such that $x_0$ belongs to $S_i.$ Since $S_i\subset F,$ we have $d_S(x,F)\leq d_S(x,S_i),$ and since $x_0\in S_i$ we have $d_S(x,x_0)\geq d_S(x,S_i)$ and thus $d_S(x,F)\geq d(x,S_i).$ Thus $d_S(x,F)=d_S(x,S_i),$ and the result follows. 

We now prove that $\ul u\leq \ol u.$ We fix $i\in I$ and we prove that $u_i\leq\ol u:$ applying a Lorentz transformation, we may assume that 
$$S_i=\{(x_1,x_2,0,\ldots,0)\in S^{n-1}:\ x_1\geq 0\}.$$ 
Let ${x'}_0\in\R^{n-2},$ and set
$$\tilde u_i(x_1,x_2):=u_i(x_1,x_2,{x'}_0)\mbox{ and }\tilde{\ol u}(x_1,x_2):=\ol u(x_1,x_2,{x'}_0).$$
Recalling the proof of Lemma \ref{lemma semitrough segment}, we observe that $\tilde u_i$ is the (scaled) semi\-trough defined Section \ref{section definition semitrough}. From Lemma \ref{comparison mean curvature} we get $H_1[\tilde{\ol u}]\leq h.$ Since $\tilde u_i$ is the semitrough with Gauss curvature equal to $k,$ from the geometric-arithmetic means inequality we get $H_1[\tilde u_i]\geq 2\sqrt{k}.$ Thus  $H_1[\tilde{\ol u}]\leq H_1[\tilde u_i].$ We suppose by contradiction that there exists $x_0=(x_1^0,x_2^0)$ such that $\tilde u_i(x_0)>\tilde{\ol u}(x_0),$ and we consider $\epsilon>0$ such that
$$\tilde u_i(x_0)>\tilde{\ol u}(x_0)+\epsilon.$$
Since $S_i$ belongs to $F,$ we have $V_F(.,x'_0)\geq V_{S_i}(.,x'_0)=V_{S^+},$ and we conclude from (\ref{asymptotic semitrough}) and (\ref{bounds upper barrier}) that
$$\liminf_{|x|\rightarrow+\infty}(\tilde{\ol u}+\epsilon)-\tilde u_i\geq \epsilon.$$
Thus the non-empty open set
$$U=\{(x_1,x_2)\in\R^2:\ \tilde u_i(x_1,x_2)>\tilde{\ol u}(x_1,x_2)+\epsilon\}$$
is bounded. Since $\tilde u_i=\tilde{\ol u}+\epsilon$ on $\partial U$ and $H_1[\tilde u_i]\geq H_1[\tilde{\ol u}+\epsilon]$ in $U,$ we get a contradiction with the maximum principle. The claim is proved.

We finally prove the strict inequality $\ul u<\ol u:$ we set $\displaystyle{\ol u_{\lambda}(x):=\frac{1}{\lambda}\ol u(\lambda x),}$ with $\lambda>1$ such that $\lambda h<2\sqrt{k}.$ Since $H_1[\ol u_{\lambda}]=\lambda h$ and $\ol u_{\lambda}>V_F,$ the arguments given in the paragraph above (with $\ol u_{\lambda}$ instead of $\ol u$) show that $\ul u\leq \ol u_{\lambda}.$ Since $\ol u_{\lambda}<\ol u,$ we obtain the result.
\end{proof}
\noindent The useful properties of the barriers are gathered in Proposition \ref{proposition intro barriers}. 
\begin{remark}\label{invariance barriers}
By construction, it is clear that if the set $F$ is contained in some affine subspace, in
$$\{(x',x'')\in\R^n=\R^k\times\R^{n-k}:\  x''=0\}$$
say, we may assume that the barriers $\ul u,\ol u$ satisfy: for all $(x',x'')\in\R^k\times\R^{n-k},$
$$\ul u(x',x'')=\ul u(x',0)\mbox{ and }\ol u(x',x'')=\ol u(x',0).$$
\end{remark}
\subsection{Construction of two auxiliary functions}
This section is devoted to the construction of auxiliary functions which are crucial for the local $C^1$ and $C^2$ estimates. The functions $\ul u,\ol u$ are the barriers constructed above. The following lemma is needed for the local $C^1$ estimate. 
\begin{lemma}\label{construction psi}
Let $K$ be a compact subset of $\R^n.$ There exists a smooth spacelike function $\psi:\R^n\rightarrow\R$ such that 
$$\psi<\ul u\mbox{ on }K\mbox{ and }\psi\geq \ol u\mbox{ near infinity.}$$
\end{lemma}
\noindent For the proof, we will need the following lemma:
\begin{lemma}
Let $F$ be a closed subset of $S^{n-1}.$ Let $\epsilon>0$ and set
$$F_{\epsilon}:=\{\xi\in S^{n-1}:\ d_S(\xi,F)\leq\epsilon\}.$$ 
The function $V_{F_{\epsilon}}-V_F$ has the following properties: 
\begin{equation}\label{property VF_epsilon 1}
\sup_{\xi\in S^{n-1}}\left|V_{F_{\epsilon}}(\xi)-V_F(\xi)\right|\leq\epsilon,
\end{equation}
and
\begin{equation}\label{property VF_epsilon 2}
\inf_{\xi\in S^{n-1}\backslash F_{\epsilon}}V_{F_{\epsilon}}(\xi)-V_F(\xi)\geq m_\epsilon,
\end{equation}
for some positive constant $m_{\epsilon}.$
\end{lemma}
\begin{proof}
By (\ref{formula VF}), for all $\xi\in S^{n-1},$
\begin{equation}\label{VF_epsilon-V_F}
V_{F_{\epsilon}}(\xi)-V_F(\xi)=\cos\left(d_S(\xi,F_{\epsilon})\right)-\cos\left(d_S(\xi,F)\right).
\end{equation}

We first prove (\ref{property VF_epsilon 1}): we observe that, for all $\xi\in S^{n-1},$
$$\left|d_S(\xi,F_{\epsilon})-d_S(\xi,F)\right|\leq \varepsilon.$$
Since $\left|\cos(\alpha)-\cos(\beta)\right|\leq|\alpha-\beta|$ for all $\alpha,\beta\in\R,$ we obtain  (\ref{property VF_epsilon 1}).

We now prove (\ref{property VF_epsilon 2}): we suppose that $\xi\notin F_{\epsilon};$ since $F\subset F_{\epsilon},$ we have
\begin{equation}\label{distance F F_epsilon}
d_S(\xi,F)=d_S(\xi,F_{\epsilon})+\epsilon.
\end{equation}
This implies in particular that
$$d_S(\xi,F_{\epsilon})\in[0,\pi-\epsilon].$$
Denoting $\alpha=d_S(\xi,F_{\epsilon}),$ we obtain from (\ref{VF_epsilon-V_F}) and (\ref{distance F F_epsilon}) that
$$V_{F_{\epsilon}}(\xi)-V_F(\xi)=\cos\alpha-\cos(\alpha+\epsilon)\geq m_{\epsilon},$$
where $\displaystyle{m_{\epsilon}=\inf_{\alpha\in[0,\pi-\epsilon]}\int_{\alpha}^{\alpha+\epsilon}\sin(t)dt}$
is positive, and we obtain (\ref{property VF_epsilon 2}).
\end{proof}
\textit{Proof of Lemma \ref{construction psi}.}
Let $K$ be a compact subset of $\R^n,$ and $R\geq 1$ be such that $K\subset \ol B_R$ (here and below $B_R$ stands for the open ball of radius $R$ in $\R^n,$ centered at the origin). We fix $\delta_0>0$ such that 
\begin{equation}\label{def delta_0}
\inf_{B_R}(\ul u-V_F)\geq\delta_0.
\end{equation}
Let $\epsilon>0$ and $F_{\epsilon}:=\{\xi\in S^{n-1}|\ d(\xi,F)\leq\epsilon\}.$ 
From (\ref{property VF_epsilon 1}) we get  
\begin{equation}
\sup_K\left|V_{F_{\epsilon}}-V_F\right|<\frac{\delta_0}{8}
\end{equation}
if $\epsilon<\frac{\delta_0}{8R}.$ Thus, if $\epsilon<\frac{\delta_0}{8R},$
\begin{equation}\label{ineq VF epsilon VF}
\sup_K\left|(V_{F_{\epsilon}}+\epsilon)-V_F\right|<\frac{\delta_0}{4}.
\end{equation}
Let $\psi$ be a spacelike function such that $\psi>V_{F_{\epsilon}}+\epsilon$ and
\begin{equation}\label{ineq VF epsilon psi}
\sup_K\left|\psi-(V_{F_{\epsilon}}+\epsilon)\right|<\frac{\delta_0}{4}.
\end{equation}
We may construct $\psi$ as follows: we first consider a spacelike function $v$ whose graph has constant mean curvature one and which is such that
$$V_{F_{\epsilon}}<v<V_{F_{\epsilon}}+c,$$
given by \cite{T} Theorem 2; here $c$ is a positive constant. We then define
$$\psi(x):=\frac{1}{\lambda}v(\lambda x)+\epsilon,$$
where $\lambda$ is a positive parameter. We have
\begin{eqnarray*}
\sup_{x\in\R^n}\left|\psi(x)-(V_{F_{\epsilon}}+\epsilon)(x)\right|&=&\sup_{x\in\R^n}\left|\frac{1}{\lambda}v(\lambda x)-V_{F_{\epsilon}}(x)\right|\\&=&\frac{1}{\lambda}\sup_{x\in\R^n}\left|v(\lambda x)-V_{F_{\epsilon}}(\lambda x)\right|\leq\frac{c}{\lambda}<\frac{\delta_0}{4}
\end{eqnarray*}
if $\lambda$ is chosen sufficiently large. From (\ref{ineq VF epsilon VF}) and (\ref{ineq VF epsilon psi}) we get 
\begin{equation}
\sup_K|\psi-V_F|<\frac{\delta_0}{2},
\end{equation}
and, from (\ref{def delta_0}), on $K,$
\begin{eqnarray*}
\ul u-\psi&\geq&\inf_K(\ul u-V_F)-\sup_K|\psi-V_F|\geq\frac{\delta_0}{2}.
\end{eqnarray*}
We now prove that there exists $r_\epsilon>0$ such that
\begin{equation}\label{existence r_epsilon}
\inf_{\R^n\backslash B_{r_\epsilon}}\left(V_{F_{\epsilon}}+\epsilon-\ol u\right)\geq\frac{\epsilon}{2}.
\end{equation}
Since $\psi>V_{F_{\epsilon}}+\epsilon,$ this will prove the last claim of the Lemma. We consider $x=r\xi\in\R^n\backslash\{0\},$ with $r>0$ and $\xi\in S^{n-1}.$ We first suppose that $\xi\in F_{\epsilon}.$ By (\ref{limsup upper barrier}) there exists $r_1,$ independent of $\xi,$ such that $\ol u(r\xi)\leq r+\frac{\epsilon}{2}$ for all $r\geq r_1.$ Thus, if $r\geq r_1,$
$$(V_{F_{\epsilon}}+\epsilon-\ol u)(r\xi)\geq(r+\epsilon)-\left(r+\frac{\epsilon}{2}\right)\geq\frac{\epsilon}{2}.$$
If we now suppose that $\xi\notin F_{\epsilon},$ we have
\begin{eqnarray*}
(V_{F_{\epsilon}}+\epsilon-\ol u)(r\xi)&\geq& \left(V_{F_{\epsilon}}-V_F\right)(r\xi)+\left(V_F-\ol u\right)(r\xi)\\
&\geq&\left(V_{F_{\epsilon}}-V_F\right)(r\xi)-c,
\end{eqnarray*}
where the constant $c$ is given by (\ref{inequality barriers}). By (\ref{property VF_epsilon 2}), $\left(V_{F_{\epsilon}}-V_F\right)(r\xi)\geq rm_{\epsilon}$ where the constant $m_{\epsilon}$ is positive. Thus, there exists $r_2$ such that if $r\geq r_2$ and $\xi\notin F_{\epsilon},$ we have
$$(V_{F_{\epsilon}}+\epsilon-\ol u)(r\xi)\geq\frac{\epsilon}{2}.$$
Taking $r_{\epsilon}=\max(r_1,r_2)$ we obtain (\ref{existence r_epsilon}).
\begin{flushright}
$\Box$
\end{flushright}
The following lemma is needed for the local $C^2$ estimate. 
\begin{lemma}\label{construction phi}
We suppose that $F$ is not included in any affine hyperplane of $\R^n,$ and we consider $K$ a compact subset of $\R^n.$ There exist a ball $B_R$ which contains $K$ and a smooth and strictly convex function $\Phi:\ol B_R\rightarrow\R$ such that
$$\Phi>\ol u\mbox{ on }K\mbox{ and }\Phi\leq\ul u\mbox{ on }\partial B_R.$$
\end{lemma}
\begin{proof}
We first note that the upper barrier $\ol u$ is strictly convex: this follows from the Splitting Theorem \cite{CT}, Theorem 3.1, together with the assumption that $F$ is not included in any affine hyperplane of $\R^n.$ Applying a affine Lorentz transformation if necessary, we may suppose that $\ol u(0)=0$ and $d\ol u_{0}=0.$ Since $\ol u$ is strictly convex, we have
\begin{equation}\label{limit upper barrier}
\lim_{|x|\rightarrow+\infty}\ol u(x)=+\infty.
\end{equation}
We fix $R'$ sufficiently large such that $K\subset \ol B_{R'}.$ We set $\Phi_0:=\sup_{B_{R'}}\ol u+1.$ Recalling (\ref{inequality barriers}),  $\ul u\geq \ol u -c$ on $\R^n.$ We thus get from (\ref{limit upper barrier}) the existence of $R>R'$ such that
\begin{equation}
\inf_{\{x:\ |x|\geq R\}}\ul u(x)\geq \Phi_0+1.
\end{equation} 
We set, for all $x\in\R^n,$
\begin{equation}
\Phi(x):=\Phi_0+\frac{1}{{R}^2}|x|^2.
\end{equation}
The function $\Phi$ is strictly convex, $\Phi\geq\ol u+1\mbox{ on }B_{R'}$ and $\Phi\leq \ul u\mbox{ on }\partial B_{R}.$
\end{proof}
\section{The construction of an entire solution of the elliptic problem}\label{section entire solution elliptic}
We assume that $F,$ $\ul u$ and $\ol u$ are as in Proposition \ref{proposition intro barriers}. The barriers $\ul u,\ol u$ are constructed in the previous section.

We first suppose that $F$ is not included in any affine hyperplane of $\R^n.$ For any positive $R,$ we set $u_R$ for the admissible solution of
$$\left\{\begin{array}{rcl}H_2[u_R]&=&1\mbox{ in }B_R\\
u_R&=&\ol u\mbox{ on }\partial B_R.
\end{array}\right.$$
This Dirichlet problem is solvable since $\ol u$ is strictly convex (by the Splitting Principle \cite{CT}, Theorem 3.1); see \cite{Bay,U}. From the Mac-Laurin inequality (\ref{mac laurin}) we get $H_1[u_R]\geq\sqrt{\frac{2n}{n-1}}.$ Thus, the comparison principle for the operator $H_1$ implies that $\ol u\geq u_R.$ Since $\ul u$ is defined as a supremum of admissible functions with scalar curvature $H_2=1,$ we also have $u_R\geq\ul u.$ Thus $u_R$ lies between the barriers, for every $R.$ The following local uniform estimates hold: for any $R_0\geq 0,$ there exist $R_1=R_1(R_0)$ sufficiently large, $\theta\in(0,1),$ and $C\geq 0$ such that: for every $R\geq R_1,$
$$\sup_{B_{R_0}}|Du_R|\leq 1-\theta\mbox{ and }\sup_{B_{R_0}}|u_R|+\sup_{B_{R_0}}|D^2u_R|\leq C.$$
For the $C^1$ local estimate, we refer to \cite{Bay2}, Proposition 4.1. The auxiliary function $\psi$ needed for the estimate is given here by Lemma \ref{construction psi}. For the local $C^2$ estimate, we refer to \cite{Bay2}, Proposition 5.1.; here is needed the auxiliary function $\Phi$ given by Lemma \ref{construction phi}. The proofs remain unchanged.

Evans-Krylov interior second derivative H\"older estimate, and Schauder interior regularity theory imply locally uniform estimates of higher derivatives. A diagonal process then yields a subsequence $u_{R_k},$ $R_k\rightarrow+\infty,$ that locally converges to a smooth solution of (\ref{entire elliptic equation}). The properties (\ref{asymptotic u along F}) and (\ref{asymptotic u on Rn}) follow from the behavior at infinity of the barriers given by (\ref{inequality barriers}) and (\ref{limsup barrier intro}).

If $F$ is included in some affine hyperplane of $\R^n,$ applying a Lorentz transformation we may suppose that $F$ belongs to 
$$S^{k-1}\times\{0\}=\{(x_1,\ldots,x_n)\in S^{n-1}:\ x_{k+1}=\cdots=x_n=0\}$$
and that $F$ is not included in any affine hyperplane of $\R^k\times\{0\},$ where $k$ belongs to $\{1,\ldots,n-1\}.$ By Remark \ref{invariance barriers}, the restrictions $\ul u_{|\R^k},$ $\ol u_{|\R^k}$ are barriers for the scalar curvature operator $H_2$ on $\R^k,$ with $\ol u_{|\R^k}$ strictly convex, and are such that (\ref{inequality barriers})-(\ref{limsup barrier intro}) hold on $\R^k.$ Thus, there exists $\tilde u:\R^k\rightarrow\R$ such that $H_2[\tilde u]=1$ and $\ul u_{|\R^k}\leq \tilde u\leq\ol u_{|\R^k}.$ The function $u$ defined on $\R^n$ by
$$u(x_1,\ldots,x_n):=\tilde u(x_1,\ldots,x_k)$$
is an entire solution of (\ref{entire elliptic equation}) such that (\ref{asymptotic u along F}) and (\ref{asymptotic u on Rn}) hold.
\section{Notation and evolution equations}\label{section evolution equations}
\subsection{Notation}
Let $\Sigma_0$ be a spacelike hypersurface of $\R^{n,1},$ and let $X:\Sigma_0\times [0,+\infty)\rightarrow\R^{n,1}$ be a family of spacelike embeddings  of $\Sigma_0$ in $\R^{n,1}:$ for every $t\geq 0,$  $\Sigma_t:=X(\Sigma_0\times\{t\})$ is a spacelike hypersurface. We set $N$ for the future oriented unit normal field of $\Sigma_t.$ We denote by $(g_{ij})$ and $(h_{ij})$ the metric and the second fundamental form induced by the Minkowski metric on the embedded hypersurface $\Sigma_t.$ We will use the Einstein summation convention, and raise or lower indices with respect to the metric $(g_{ij}).$ The components of the curvature endomorphism are thus denoted by $h^i_j,$ and we will often write problem (\ref{entire parabolic equation}) in the equivalent form
\begin{equation}\label{parabolic Dirichlet problem X}
\left\{\begin{array}{rcl}
\dot X&=&(F((h^i_j)_{i,j})-\hat f(X,t))N\mbox{ in }\Sigma_0\times(0,+\infty)\\
X(.,0)&=&X_0\mbox{ on }\Sigma_0,
\end{array}\right.
\end{equation}
where $X_0$ is the canonical embedding of $\Sigma_0,$ $F(A)$ is the square root of the sum of the principal minors of order $2$ of the matrix $A,$ and $\hat f$ is a positive function on $\R^{n,1}\times[0,+\infty)$ (constant equal to one in (\ref{entire parabolic equation})). Let
$$F_i^j:=\frac{\partial F}{\partial h_j^i}\left((h^i_j)_{i,j}\right).$$
If $(h^i_j)_{i,j}$ is diagonal, so is $(F_i^j)_{i,j},$ and $F_i^i=\frac{1}{2F}\sigma_{1,i}$ for all $i,$ 
where 
$$\sigma_{1,i}=\sum_{k,k\neq i}\lambda_k.$$ Here and below we denote by $\lambda_1\geq\cdots\geq\lambda_n$ the principal curvatures of $\Sigma_t.$ $(F_i^j)_{i,j}$ defines a $(1,1)$ tensor on $\Sigma_t.$ Raising the index $i$ we also will use the symmetric tensor $(F^{ij})_{i,j}.$ Analogously we define
$$F^{ij,kl}:=g^{ii'}g^{kk'}\frac{\partial^2 F}{\partial h_j^{i'}\partial h_l^{k'}}\left((h^i_j)_{i,j}\right).$$
We say that $X$ solution of (\ref{parabolic Dirichlet problem X}) is admissible if, for every $t\geq 0,$  $\Sigma_t$ is an admissible hypersurface, which means that  ${H_1}({\Sigma_t})>0$ and ${H_2}({\Sigma_t})>0.$ Admissibility of a solution of (\ref{entire parabolic equation}) is defined similarly. We denote by $D$ the usual covariant derivative on $\R^{n,1}$ (or on $\R^n$), $\nabla$ the covariant derivative induced on $\Sigma_t,$ and use a semi-colon to denote the components of covariant derivatives on $\Sigma_t.$ Finally, the Minkowski metric on $\R^{n,1}$ is denoted by $\langle.,.\rangle,$ the Minkowski norm of spacelike vectors of $\R^{n,1}$ by $|.|,$ and  the usual euclidian norm by $|.|_{eucl}.$ 
\subsection{Evolution equations}
For a hypersurface moving according to 
$$\dot X=(F(h^i_j)-\hat f(X,t))N,$$
we have 
\begin{equation}\label{evolution gij}
\frac{d}{dt}g_{ij}=2(F-\hat f)h_{ij},
\end{equation}
\begin{equation}\label{evolution F-hat f}
\frac{d}{dt}(F-\hat{f})-F^{ij}(F-\hat f)_{ij}=-F^{ij}h_{ik}h^k_j(F-\hat f)-N^{\alpha}\hat f_{\alpha}(F-\hat f)-\hat f_t,
\end{equation}
\begin{equation}\label{evolution hij}
\begin{array}{lcl}
\frac{d}{dt}h_{ij}-F^{kl}h_{ij;kl}&=&F(h_i^ah_{aj})-h_{ij}F^{kl}(h_k^ah_{al})\\&&+F^{kl,pq}h_{kl;i}h_{pq;j}-\hat f_{ij}+(F-\hat f)h_i^kh_{kj},
\end{array}
\end{equation}
and, defining $u:=-\langle e_{n+1},X\rangle$ and  $\nu:=-\langle e_{n+1},N\rangle,$ 
\begin{equation}\label{evolution height u}
\dot u-F^{ij}u_{ij}=-\hat f\nu,
\end{equation}
and 
\begin{equation}\label{evolution nu time xn+1}
\dot \nu-F^{ij}\nu_{ij}=-\nu F^{ij}h_j^kh_{ki}+\hat f_jt^j,
\end{equation}
where the $t^j$'s are the coordinates of the component tangential to $\Sigma_t$ of $e_{n+1}.$ For the proofs we refer to \cite{Ge1} and \cite{BaySch}, where similar evolution equations are obtained.
\section{The parabolic Dirichlet problem}\label{section parabolic Dirichlet problem}
The aim of this section is to prove the following 
\begin{theorem}\label{theorem parabolic Dirichlet problem} Let $\Omega$ be a uniformly convex bounded domain in $\R^n$ with smooth boundary, let $u_0:\overline{\Omega}\rightarrow\R$ be a smooth, spacelike and strictly convex function, and let $\hat f:\ol{\Omega}\times\R\times[0,+\infty)\rightarrow (0,+\infty),$ $(x,u,t)\mapsto\hat f(x,u,t)$ be a smooth positive function such that $\hat f_t\leq 0.$ We suppose that, for all $x\in\Omega,$
$$H_2[u_0]^{\frac{1}{2}}(x)-\hat f(x,u_0(x),0)\geq 0.$$
Then the parabolic Dirichlet problem 
\begin{equation}\label{parabolic Dirichlet problem T} 
\left\{\begin{array}{rcl}
-\frac{\dot u}{\sqrt{1-|Du|^2}}+H_2[u]^{\frac{1}{2}}& = & \hat f(x,u,t)\mbox{ in }\Omega\times (0,+\infty)\\
u(x,t)& = &u_0(x)\mbox{ on }\partial\Omega\times[0,+\infty)\cup\Omega\times\{0\},
\end{array}\right.
\end{equation}
has an admissible solution $u\in C^{\infty}(\ol\Omega\times [0,+\infty))$ if, on the corner of the parabolic domain, the compatibility conditions of any order are satisfied.
\end{theorem}
At the boundary, compatibility conditions of any order are fulfilled, so we get a smooth admissible solution for a short time interval. 
 
We consider $T$ maximal such that the parabolic Dirichlet problem (\ref{parabolic Dirichlet problem T}) has an admissible solution on $\ol\Omega\times [0,T),$ and suppose by contradiction that $T<+\infty.$
We need the following a priori estimates:
\begin{equation}\label{a priori estimates Dirichlet parabolic}
\sup_{\ol\Omega\times[0,T)}|Du|\leq 1-\theta,\ \sup_{\ol\Omega\times[0,T)}|H_2[u]^{\frac{1}{2}}-\hat f|\leq C_1,\ \sup_{\ol\Omega\times[0,T)}|D^2u|\leq C_2,
\end{equation}
and
\begin{equation}\label{estimate admissible}
\alpha_0\leq\inf_{\ol\Omega\times[0,T)}H_2[u]^{\frac{1}{2}}.
\end{equation}
with $\theta\in(0,1],$ $C_1,C_2\geq 0,$ and $\alpha_0>0.$ With these estimates at hand (and the obvious $C^0$ estimate), the estimates of Krylov and Safonov and the Schauder theory imply estimates of higher derivatives of $u.$ We may thus extend $u$ to a solution on $[0,T].$ Admissibility at the time $T$ is guaranteed by (\ref{estimate admissible}) and Mac-Laurin inequality (\ref{mac laurin}). The short time existence theory then yields a solution on $[0,T+\epsilon),$ $\epsilon>0,$ and thus a contradiction with the definition of $T.$  

In the rest of the section we carry out estimates (\ref{a priori estimates Dirichlet parabolic}) and (\ref{estimate admissible}). Instead of (\ref{parabolic Dirichlet problem T}) we will also consider the equivalent problem
\begin{equation}\label{parabolic Dirichlet problem T   X}
\left\{\begin{array}{rcl}
\dot X&=&(F-\hat f)N\mbox{ in }\Sigma_0\times(0,T)\\
X&=&X_0\mbox{ on }\partial\Sigma_0\times[0,T)\cup\Sigma_0\times\{0\},
\end{array}\right.
\end{equation}
where $\Sigma_0=\graph u_0$ and $X:\Sigma_0\times[0,T)\rightarrow\R^{n,1}$ denotes the embedding vector in $\R^{n,1}.$ 

We will denote by $D_0$ the domain of dependence in $\R^{n,1}$ of the boundary data $\Sigma_0=\graph u_0$ (a point $p$ belongs to $D_0$ if every non-spacelike ray through $p$ intersects $\Sigma_0$). $D_0$ is a compact subset of $\R^{n,1},$ and, since $X=X_0$ on $\partial\Sigma_0\times[0,T)$ and $X$ is spacelike, $X(x,t)$ belongs to $D_0$ during the evolution.

\subsection{The $C^1$ estimate}\label{subsection C1 estimate}. 

\subsubsection{The maximum principle for the first derivatives}
\begin{proposition} \label{C1 estimate maximum principle}
Let $X:\Sigma_0\times[0,T)\rightarrow\R^{n,1}$ be a smooth solution of (\ref{parabolic Dirichlet problem T X}). Then
$$\sup_{\Sigma_0\times[0,T)}\nu\leq C,$$
where $C$ depends on the $C^0$ estimate, on $\displaystyle{\sup_{D_0\times[0,T]}|D\log\hat f|_{eucl}}$ and on an upper bound of $\nu$ on the parabolic boundary $\partial\Sigma_0\times[0,T)\cup\Sigma_0\times\{0\}.$
\end{proposition}
\begin{proof}
Let $K$ be a positive constant to be chosen later. At an interior maximum of $\psi=e^{Ku}\nu,$ we have
\begin{equation*}
\dt(\log\psi)-F^{ij}(\log\psi)_{ij}\geq 0.
\end{equation*}
Thus, using (\ref{evolution height u}),
\begin{equation}\label{condition max psi}
\dt(\log\nu)-F^{ij}(\log\nu)_{ij}-K\hat f\nu\geq 0.
\end{equation}
Moreover, we have
$$\dt(\log\nu)-F^{ij}(\log\nu)_{ij}=\frac{1}{\nu}\left(\dot\nu-F^{ij}\nu_{ij}\right)+\frac{1}{\nu^2}F^{ij}\nu_i\nu_j,$$
with
\begin{eqnarray*}
\dot\nu-F^{ij}\nu_{ij}&=&\hat f_jt^j-\nu F^{ij}h_j^kh_{ki}\\
&\leq&\nu|\nabla\hat f|-\nu \sum_{i=1}^nF^{ii}\lambda_i^2,
\end{eqnarray*}
and $\nu_i=u_i\lambda_i.$ Here the tensors are written in an orthonormal basis of principal directions, and we used that $g_{ij}t^it^j\leq \nu^2.$ Thus, (\ref{condition max psi}) implies
$$\left(\sum_{i=1}^nF^{ii}\lambda_i^2-\sum_{i=1}^nF^{ii}\lambda_i^2\frac{u_i^2}{\nu^2}\right)+K\hat f\nu\leq|\nabla\hat f|.$$
Discarding the first term which is positive, and using $|\nabla\hat f|\leq \nu|D\hat f|_{eucl},$ we obtain $K\hat f\leq |D\hat f|_{eucl},$ which is impossible for $K$ sufficiently large such that
\begin{equation}\label{condition K}
K>\sup_{D_0\times[0.T]}|D\log\hat f|_{eucl}.
\end{equation}
Thus, if $K$ satisfies (\ref{condition K}), the function $\psi=e^{Ku}\nu$ reaches its maximum on the parabolic boundary of $\Sigma_0\times[0,T),$ and the result follows.
\end{proof}
\subsubsection{The $C^1$ estimate at the boundary} 
\begin{proposition}\label{C1 estimate boundary}
Let $u:\ol\Omega\times [0,T)\rightarrow\R$ be an admissible solution of (\ref{parabolic Dirichlet problem T}). Then there exists $\theta\in (0,1]$ such that
$$\sup_{\partial\Omega\times[0,T)}|Du|\leq 1-\theta.$$
The number $\theta$ depends on $\inf_{D_0\times[0,T]}\hat f,$ $\sup_{D_0\times[0,T]}\hat f,$ and $\sup_{\ol\Omega}|Du_0|.$ 
\end{proposition}
\begin{proof}
We fix $x_0\in\partial\Omega,$ and we denote by $n$ the inner normal of $\partial\Omega$ at $x_0.$ We define, for $c(n)=\sqrt{\frac{n-1}{2n}},$
\begin{equation}\label{definition P1 P2}
P_1[u]=-\frac{\dot u}{\sqrt{1-|Du|^2}}+c(n)H_1[u],\ P_2[u]=-\frac{\dot u}{\sqrt{1-|Du|^2}}+H_2[u]^{\frac{1}{2}}.
\end{equation}
\textit{The construction of the upper barrier.}
Let $u_1$ be a spacelike function such that $u_1\geq u_0$ in $\overline{\Omega},$ $u_1=u_0$ on $\partial\Omega,$ and $c(n)H_1[u_1]\leq\inf_{D_0\times[0,T]}\hat f.$ We may take for $u_1$ the solution of $H_1[u_1]=c$ in $\Omega,$ $u_1=u_0$ on $\partial\Omega,$ where $c$ is a small constant; this Dirichlet problem is solved in \cite{BS}, Theorem 4.1. Defining $u_1(x,t):=u_1(x),$ we have $P_1[u_1]\leq P_1[u]$ (since $P_1[u]=\hat f+c(n)H_1[u]-H_2[u]^{\frac{1}{2}}$ and using (\ref{mac laurin})), $u_1\geq u$ on the parabolic boundary, and thus, by the maximum principle,  $u_1\geq u$ on $\ol\Omega\times[0,T).$ Since $u_1(x_0,t)=u(x_0,t)=u_0(x_0)$ for all $t\in[0,T),$  we obtain
$$\partial_nu(x_0,t)\leq\partial_nu_1(x_0),\ \forall t\in[0,T).$$
\textit{The construction of the lower barrier.}
Let $u_2$ be an admissible function such that $u_2\leq u_0$ in $\overline{\Omega},$ $u_2=u_0$ on $\partial\Omega,$ and $H_2[u_2]^{\frac{1}{2}}\geq\sup_{D_0\times[0,T]}\hat f.$ We may take for $u_2$ the strictly convex and spacelike solution of $K[u_2]=c$ in $\Omega,$ $u_2=u_0$ on $\partial\Omega,$ where $c$ is a large constant; $K[u_2]$ stands for the Gauss curvature of $\graph u_2;$ this Dirichlet problem is solved in \cite{De}. Defining $u_2(x,t):=u_2(x),$ we have $P_2[u_2]\geq P_2[u]$ on $\Omega\times(0,T),$ $u_2\leq u$ on the parabolic boundary, and thus, by the maximum principle, $u_2\leq u$ on $\ol\Omega\times[0,T).$ Since $u_2(x_0,t)=u(x_0,t)=u_0(x_0)$ for all $t\in[0,T),$ we obtain
$$\partial_nu(x_0,t)\geq\partial_nu_2(x_0),\ \forall t\in[0,T).$$
Finally, since the tangential derivatives at the boundary of $u,u_1$ and $u_2$ coincide, we get
$$\sup_{\partial\Omega\times[0,T)}|Du|\leq\max(\sup_{\partial\Omega}|Du_1|,\sup_{\partial\Omega}|Du_2|),$$
and the result follows.
\end{proof}
\subsection{The velocity estimate}\label{subsection velocity estimate}
\begin{proposition}\label{lemma velocity bound}
Let $X$ be an admissible solution of the parabolic Dirichlet problem  (\ref{parabolic Dirichlet problem T X}). We recall that $\hat f_t\leq 0$ and we suppose that $C$ and $K$ are two positive constants such that
\begin{equation}\label{bounds f derivatives}
|\hat f_t|\leq C\mbox{ and }\left|N^{\alpha}\hat f_{\alpha}\right|\leq K
\end{equation}
during the evolution.  If $F-\hat f\geq 0$ at $t=0$ then, for all $t\in[0,T),$
\begin{equation}\label{velocity bound}
0\leq F-\hat f\leq e^{Kt}\sup_{t=0}(F-\hat f)+\frac{C}{K}(e^{Kt}-1).
\end{equation}
In particular, there exist two constants $\alpha_0,\beta_0>0$ such that
\begin{equation}\label{F bounded}
\alpha_0\leq F\leq \beta_0
\end{equation}
on $\Sigma_0\times[0,T).$ The constants $\alpha_0,\beta_0$ depend on $C,K,T,$ $\sup_{t=0}(F-\hat f),$ and $\inf_{D_0\times[0,T]}\hat f.$ 
\end{proposition}
\begin{remark}
The constants $C$ and $K$ in (\ref{bounds f derivatives}) are controlled by $\sup_{D_0\times[0,T]}|\hat f_t|,$\\ $\sup_{D_0\times[0,T]}|D\hat f|_{eucl},$ and by the $C^1$ estimate obtained Section \ref{subsection C1 estimate}. 
\end{remark}
\begin{proof}
We first consider $\displaystyle{\Psi_1:=(F-\hat f)e^{-Kt}.}$ The evolution equation of $\Psi_1$ is
\begin{eqnarray*}
\dot\Psi_1-F^{ij}{\Psi_1}_{ij}&=&\left[\dt(F-\hat f)-F^{ij}(F-\hat f)_{ij}-K(F-\hat f)\right]e^{-Kt}\\
&=&\left[-F^{ij}h_{ik}h_j^k(F-\hat f)-N^{\alpha}\hat f_{\alpha}(F-\hat f)-\hat f_t-K(F-\hat f)\right]e^{-Kt}\\
&\geq&\left[-F^{ij}h_{ik}h_j^k-N^{\alpha}\hat f_{\alpha}-K\right]\Psi_1,
\end{eqnarray*}
since $\hat f_t\leq 0.$ Let $T_1\in(0,T).$ The function $\Psi_1$ on $\ol{\Omega}\times[0,T_1]$ reaches its minimum at some point $(x_0,t_0).$ Assume that $(x_0,t_0)\in\Omega\times(0,T_1]$ and that $\Psi_1(x_0,t_0)<0.$ At $(x_0,t_0),$ we have $\displaystyle{\dot\Psi_1-F^{ij}{\Psi_1}_{ij}\leq 0,}$ which gives
\begin{equation}\label{inequation velocity bound1}
-F^{ij}h_{ik}h_j^k-N^{\alpha}\hat f_{\alpha}-K\geq 0
\end{equation}
and a contradiction with (\ref{bounds f derivatives}). Since $\Psi_1\geq 0$ on the parabolic boundary, we conclude that $\Psi_1\geq 0$ on $\overline\Omega\times[0,T),$ and thus that $\displaystyle{F-\hat f\geq 0}$ on $\displaystyle{\overline\Omega\times[0,T).}$

We now consider $\displaystyle{\Psi_2:=\left(F-\hat f+\frac{C}{K}\right)e^{-Kt},}$ whose evolution equation is
\begin{equation*}
\dot\Psi_2-F^{ij}{\Psi_2}_{ij}=\left[-F^{ij}h_{ik}h_j^k(F-\hat f)-N^{\alpha}\hat f_{\alpha}(F-\hat f)-\hat f_t-K(F-\hat f)-C\right]e^{-Kt}.
\end{equation*}
Let $T_2\in(0,T).$ The function $\Psi_2$ on $\ol{\Omega}\times[0,T_2]$ reaches its maximum at some point $(x_0,t_0).$ Assume that $(x_0,t_0)\in\Omega\times(0,T_2]$ and that $\Psi_2(x_0,t_0)>\frac{C}{K}.$ The latter implies that $F-\hat f>0$ at $(x_0,t_0).$ At $(x_0,t_0),$ $\displaystyle{\dot\Psi_2-F^{ij}{\Psi_2}_{ij}\geq 0,}$ which gives
\begin{equation}\label{inequation velocity bound2}
-F^{ij}h_{ik}h_j^k(F-\hat f)-N^{\alpha}\hat f_{\alpha}(F-\hat f)-\hat f_t-K(F-\hat f)-C\geq 0
\end{equation}
and a contradiction with (\ref{bounds f derivatives}). Since $\Psi_2\leq\frac{C}{K}$ on $\partial\Omega\times[0,T)$ and $\Psi_2\geq\frac{C}{K}$ on $\Omega\times \{0\},$ we conclude that, for all $(x,t)\in\overline\Omega\times[0,T),$
$$\Psi_2(x,t)\leq\sup_{x\in\ol\Omega}\Psi_2(x,0),$$
which proves the proposition.
\end{proof}
\subsection{The $C^2$ estimate}
\subsubsection{The maximum principle for the second derivatives}\label{Dirichlet problem C2 max pple}
\begin{proposition}
Let $X:\Sigma_0\times[0,T)\rightarrow\R^{n,1}$ be a solution of (\ref{parabolic Dirichlet problem T X}). Then
$$\sup_{\xi\in T\Sigma_t,\ |\xi|=1}h_{ij}\xi^i\xi^j\leq C$$
during the evolution, for some constant $C$ which depends on a bound of the second fundamental form on the parabolic boundary $\partial\Sigma_0\times[0,T)\cup\Sigma_0\times\{0\}$ and on estimates obtained before.
\end{proposition}
\begin{proof}
$ $The estimate relies on J.Urbas $C^2$ estimate \cite{U} for the elliptic Dirichlet problem. We fix $T_1\in(0,T).$ For $t\in[0,T_1],$ $x\in\Sigma_0,$ and $\xi\in T_{X(x,t)}\Sigma_t$ with $|\xi|=1,$ we consider
$$\tilde{W}(x,\xi,t):=\eta^{\beta}(X(x,t))h_{ij}\xi^i\xi^j,$$
where $\eta$ is a positive function on $\R^{n,1}$ and $\beta$ is a positive constant, which will be defined later. We suppose that  $\tilde{W}$ reaches  its  maximum at $(x_0,\xi_0,t_0)$ where $t_0\in(0,T_1],$ $x_0$ is an interior point of $\Sigma_{0},$ and $\xi_0\in T_{X(x_0,t_0)}\Sigma_{t_0},$ with $|\xi_0|=1.$ We choose ${e_1}^0,\ldots,{e_n}^0$ a local frame on $\Sigma_0$ which induces on $\Sigma_{t_0}$ an orthonormal frame $\hat{e_1},\ldots,\hat{e_n}$ such that 
$$\hat{e_1}(X(x_0,t_0))=\xi_0\mbox{ and }\nabla_{\hat{e_i}}\hat{e_j}(X(x_0,t_0))=0.$$
Observe that $\hat{e_1}$ is a principal direction of $\Sigma_{t_0}$ at $X(x_0,t_0),$ associated to the largest principal curvature. We still denote by $\hat{e_1},\ldots,\hat{e_n}$ the frame induced by ${e_1}^0,\ldots,{e_n}^0$ on $\Sigma_t,$ for every $t.$ The function
$$W(x,t):=\eta^{\beta}\ff(\hat{e_1},\hat{e_1})/|\hat{e_1}|^2,$$
where $II$ stands for the second fundamental form of $\Sigma_t,$ reaches its maximum at $(x_0,t_0);$ we get
$$\begin{array}{lll}
\dt {\log W}-F^{ij}{(\log W)}_{ij}&=&\beta\left(\dot{\log\eta}-F^{ij}(\log\eta)_{ij}\right)+\frac{1}{h_{11}}\left(\dot h_{11}-F^{ij}h_{11;ij}\right)\\&&+\frac{1}{h_{11}^2}F^{ij}h_{11;i}h_{11;j}-\frac{\dot g_{11}}{g_{11}}\geq 0.
\end{array}$$
From the evolution equation (\ref{evolution hij}), we get:
$$\begin{array}{rcl}
\dot h_{11}-F^{ij}h_{11;ij}&=&Fh_{11}^2-h_{11}F^{kl}(h_k^ah_{al})\\&&+F^{kl,pq}h_{kl;1}h_{pq;1}-\hat f_{11}+(F-\hat f)h_{11}^2,
\end{array}$$
and thus:
\begin{equation}\label{inequality fundamental C2 interior estimate}
\begin{array}{rcl}
\beta\left(\dot{\log\eta}-F^{ij}(\log\eta)_{ij}\right) +Fh_{11}-F^{kl}(h_k^ah_{al})&&\\+\frac{1}{h_{11}}F^{kl,pq}h_{kl;1}h_{pq;1}-\frac{\hat f_{11}}{h_{11}}+(F-\hat f)h_{11} &&\\+\frac{1}{h_{11}^2}F^{ij}h_{11;i}h_{11;j}-\frac{\dot g_{11}}{g_{11}}&\geq& 0.
\end{array}
\end{equation}
Recalling (\ref{F bounded}), $|F-\hat f|$ is under control, and we have the following estimates (where the largest principal curvature is denoted by $\lambda_1$, and $C_1,C_2,C_3$ are constants under control): 
$$Fh_{11}+(F-\hat f)h_{11}\leq C_1(1+\lambda_1),$$
\begin{equation}\label{estimate hat f 11}
|\hat f_{11}|\leq C_2(1+\lambda_1),
\end{equation}
and, from (\ref{evolution gij}),
$$\left|\frac{\dot g_{11}}{g_{11}}\right|\leq C_3\lambda_1.$$
For estimate (\ref{estimate hat f 11}), we refer to \cite{U} p312-313. Thus
\begin{equation}\label{inequality fundamental C2 interior estimate reduced}
\begin{array}{rcl}
0&\geq&-\beta\left(\dot{\log\eta}-F^{ij}(\log\eta)_{ij}\right) +F^{kl}(h_k^ah_{al})-C_4(1+\lambda_1)\\&&-\frac{1}{h_{11}}F^{kl,pq}h_{kl;1}h_{pq;1}-\frac{1}{h_{11}^2}F^{ij}h_{11;i}h_{11;j},
\end{array}
\end{equation}
for some constant $C_4.$ We choose $\log\eta=\Phi,$ with
$$\Phi(x_1,\ldots,x_n,x_{n+1}):=\varphi(x_1,\ldots,x_n),$$
where $\varphi$ is some strictly convex function on $\R^n.$ Note that 
\begin{equation}\label{dot Phi}
\dot{\log\eta}=(F-\hat f)d\Phi_{X(x_0,t_0)}(N)
\end{equation}
is under control. Moreover, following J.Urbas \cite{U}, page 313,
 $$ F^{ij}(\log\eta)_{ij}\geq C_0\tau-C$$
 where $\tau=\sum_iF^{ii},$ and $C_0,C$ are constants under control. Finally,
 $$\begin{array}{rcl}
 0&\geq& \beta(C_0\tau-C')+F^{kl}(h_k^ah_{al})-C_4(1+\lambda_1)\\
 &&-\frac{1}{\lambda_1}F^{kl,pq}h_{kl;1}h_{pq;1}-\frac{1}{\lambda_1^2}F^{ij}h_{11;i}h_{11;j},
\end{array}$$
where $C'$ is under control, which is analogous to inequality (2.8) obtained by J.Urbas  in \cite{U} page 312. We then obtain the estimate of $\lambda_1$ following  the arguments used in \cite{U} p. 314-315, without any modification. This gives the $C^2$ estimate if the $C^2$ estimate at the parabolic boundary is known.  
\end{proof}

\subsubsection{The $C^2$ estimate at the boundary}
\begin{proposition}
Let $u:\overline{\Omega}\times[0,T)\rightarrow\R$ be a smooth solution of the parabolic Dirichlet problem (\ref{parabolic Dirichlet problem T}). Then
$$\sup_{(x,t)\in\partial\Omega\times(0,T)}|D^2u(x,t)|\leq C,$$
for some constant $C$ under control (which depends on the estimates obtained Sections \ref{subsection C1 estimate} and \ref{subsection velocity estimate}).
\end{proposition}
We fix $(x_0,t_0)\in\partial\Omega\times(0,T).$ Following \cite{IL2}, we write the evolution equation (\ref{parabolic Dirichlet problem T}) on the form 
$$-\dot u+\tilde{F}(Du,D^2u)=\hat f(x,u,t)\times\gamma(Du),$$
where $\gamma(Du)=\sqrt{1-|Du|^2},$ and $\tilde{F}(Du,D^2u)$ is the square root of the second elementary symmetric function of the principal values of the $n\times n$ matrix whose coefficient $(i,j)$ is given by
\begin{equation}\label{matrix tilde}
\sum_{k=1}^n\left(\delta_{ik}+\frac{u_iu_k}{1-|Du|^2}\right)u_{kj}.
\end{equation}
Let us consider the linear operator
$$LW:=-\dot W+\tilde{F}^{ij}W_{ij},$$
where $\tilde{F}^{ij}=\frac{\partial \tilde{F}}{\partial q_{ij}}(Du,D^2u).$ We suppose that $(e_1,\ldots, e_n)$ is a basis such that $(e_1,\ldots,e_{n-1})$ is an orthonormal basis made of principal vectors of $\partial\Omega$ at $x_0$ and $e_n$ is the unit inner normal of $\Omega$ at $x_0.$ Let
$$W(x,t):=g(x,Du(x,t))-\frac{K}{2}\sum_{k=1}^{n-1}(u_k(x,t)-u_k(x_0,t_0))^2,$$
where $g:\overline{\Omega}\times B(0,1)\rightarrow\R$ is a given smooth function and $K$ is a constant.
\\

The following key inequality is a lorentzian analog of (2.4) in \cite{IL2}: 
\begin{lemma}
If $K$ is sufficiently large under control, $W$ satisfies
\begin{equation}\label{bound LW}
LW\leq C_1(1+|DW|+\sum_{ij} {\tilde F}^{ij}W_iW_j+\sum_i{\tilde F}^{ii}),
\end{equation}
where $C_1$ depends on the $C^1$ estimate and on the constants $\alpha_0,\beta_0$ in (\ref{F bounded}).
\end{lemma}
\textit{Sketch of the proof:} following the lines of \cite{IL2} Section 3, we obtain the following expression  for $LW,$ which is analogous to (3.13) in \cite{IL2}:
\begin{equation}\label{expression LW}
LW=-K\sum_{\alpha=1}^n\tilde\sigma_{1,\alpha}u_{\alpha\alpha}^2\left(\sum_{k=1}^{n-1}{\eta^\alpha_k}^2\right)+j_1+j_2+j_3,
\end{equation}
with 
$$j_1=-2\sum_{\alpha=1}^n\tilde\sigma_{1,\alpha}u_{\alpha}u_{\alpha\alpha}W_{\alpha},$$
$$|j_2|\leq C\left(\sum_{\alpha=1}^n\tilde\sigma_{1,\alpha}|u_{\alpha\alpha}|+\sum_{\alpha=1}^n\tilde\sigma_{1,\alpha}\right)\mbox{ and }|j_3|\leq C\left(1+|DW|\right),$$
where $C$ is a constant under control. Here we use the letter $\alpha$ for derivatives in a basis $(\tau_{\alpha})$ of $\R^n$ which induces by the map $x\mapsto(x,u(x,t))$ an orthonormal basis of principal vectors of $\graph u$ at $(x,u(x,t));$ moreover $\tilde\sigma_{1,\alpha}$ denotes a sum of principal values of (\ref{matrix tilde}), and the numbers $\eta^\alpha_k$ are such that $e_k=\sum_{\alpha}\eta^\alpha_k\tau_{\alpha},$ for $k=1,\ldots,n.$ Expression (\ref{expression LW}) is also analogous to (26) in \cite{Bay}. We then follow the arguments in \cite{Bay}, from page 19 to page 23, without modification, and obtain (\ref{bound LW}).$\Box$
\\

\noindent Setting
$$\tilde{W}(x,t):=\exp\left(-C_1g(x_0,Du(x_0,t_0))\right)-\exp(-C_1W)-b|x-x_0|^2,$$
the following holds (see \cite{IL2}, inequality (2.5)) :
\begin{lemma}
If $b$ is sufficiently large,
$$L\tilde{W}\leq C_2(1+|D\tilde{W}|),$$
where $C_2$ is some constant under control.
\end{lemma}
We first estimate the mixed second derivatives. Let $e_s$ be a unit vector tangent to $\partial\Omega$ at $x_0,$ and let $\xi$ be the local vector field tangent to the boundary, and spanned by the vector $e_s:$ for all $x\in\ol{\Omega}\cap\ol{B_r}(x_0),$
$$\xi(x):=e_s+\rho_s(x')e_n,$$
where $e_n$ is the inner normal vector of $\Omega$ at $x_0,$ and where, in the splitting $\R^n=T_{x_0}\partial\Omega\oplus\R.e_n,$ $x-x_0=(x',x_n)$ and $\partial\Omega$ is locally the graph of $\rho.$ 

As in \cite{Bay}, we take $g(x,p)=\langle p,\xi(x)\rangle,$ and we define the barrier function
$$v=-a_0|x-x_0|^2-h(d)+\psi(x'),$$
with $h(d)=c_0(1-e^{-b_0d})$ and
$$\begin{array}{r}
\psi(x')=\exp(-C_1{u_0}_{\xi}(x_0))-\exp(-C_1{u_0}_{\xi}(x'))\exp\left(C_1K\sum_{k=1}^{n-1}({u_0}_k(x')-{u_0}_k(x_0))^2\right.\\\left.+2C_1K(|{u_0}_n(x')|^2+1)|D\rho(x')|^2\right).
\end{array}$$
Here $d$ denotes the distance function to the boundary-point $x_0,$ ${u_0}_{\xi}(x)$ denotes $\langle Du_0(x),\xi(x)\rangle,$ and, for a function $f$ defined on $\R^n,$  an expression like $f(x')$ stands for $f(x',\rho(x')).$

We first verify that $v\leq\tilde{W}$ on the parabolic boundary of $\overline{\Omega}\cap \overline{B_r}(x_0)\times [0,T)$:
\\- on $\partial(\overline{\Omega}\cap \overline{B_r})\times[0,T):$ $v\leq\tilde{W}$ on $\partial\Omega\cap\overline{B_r}\times[0,T)$ by the very definition of $\psi,$ and on $\overline{\Omega}\cap\partial\overline{B_r}\times[0,T)$ if $a_0=a_0(r,\sup{|\psi|},\sup|\tilde{W}|)$ is chosen sufficiently large. 
\\- on $\left(\overline{\Omega}\cap \overline{B_r}\right)\times\{0\}:$ we suppose that $x_0=0,$ we fix $x\in \overline{\Omega}\cap \overline{B_r},$ and we consider $\omega(s)=\tilde{W}(sx,0)-v(sx),\ s\in[0,1]$ ($\Omega$ is convex). We have $\omega(0)\geq 0.$ Setting 
$$C=\sup_{\overline{\Omega}\cap \overline{B_r}\times\{0\}}|D\tilde{W}|+\sup_{\overline{\Omega}\cap \overline{B_r}\times\{0\}}|D\psi|,$$ 
we see by a direct computation that, for all $s\in[0,1],$
$$\omega'(s)\geq|x|\{b_0c_0e^{-b_0r}-C\}.$$ 
We thus obtain the property if $c_0=c_0(b_0)$ is chosen large such that 
\begin{equation}\label{condition c_0}
b_0c_0e^{-b_0r}\geq C.
\end{equation}
From the proof of Lemma 4.6. in \cite{Bay}, we see that 
$$Lv> C_2(1+|Dv|)$$
 on $\Omega\cap B_r(x_0)\times (0,T),$ if $h'=b_0c_0e^{-b_0d},-\frac{h''}{h'}=b_0$ are large, which is compatible with (\ref{condition c_0}). The comparison principle implies that $v\leq\tilde{W}$ on the parabolic domain, and, since $v(x_0)=\tilde{W}(x_0,t_0)=0,$ we get
$$v_n(x_0)\leq\tilde{W}_n(x_0,t_0),$$
which gives the estimate
$$u_{s n}(x_0,t_0)\geq C_3,$$
where $C_3$ is a controlled constant. As in \cite{Bay}, to estimate $u_{sn}(x_0,t_0)$ from above we do the same with $g(x,p)=-\langle p,\xi(x)\rangle=-p_s-\rho_s(x')e_n.$
\\

We now estimate the double normal derivatives. A lower bound follows from $H_1[u]>0$ and from the estimates of tangential and mixed second derivatives. To obtain an upper bound, we use a technique of N.S.Trudinger \cite{T2,ILT}. We used this method in \cite {Bay} for the elliptic Dirichlet problem. We define, for $p'\in B(0,1)\subset\R^{n-1}$ and $q'$  a $(n-1)\times (n-1)$ symmetric matrix
$$\tilde{F_1}(p',q')=\sum_{i,j=1}^{n-1}\left(\delta_{ij}+\frac{p'_ip'_j}{1-|p'|^2}\right)q'_{ij}.$$
We denote by $\gamma$ the outward unit normal to $\partial\Omega$ in $\R^n,$ and by $\partial$ the tangential differential operator on $\partial\Omega.$ We fix $T_1\in (0,T),$ and $(y,t_1)\in\partial\Omega\times[0,T_1]$ such that $\tilde{F_1}(\partial u_0,\partial^2 u_0+u_\gamma\partial\gamma)$ reaches its minimum at $(y,t_1).$ As in \cite{Bay} page 27, an upper bound of $u_{nn}$ on $\partial\Omega\times[0,T_1]$ follows from an upper bound of $u_{nn}(y,t_1).$ We keep the notation introduced above, but here adapted to the boundary-point $y.$
We set
$$g(x,p)=\tilde{F_1}(\partial u_0(x'),\partial^2 u_0(x')+\langle p,\gamma(x')\rangle\partial\gamma(x')),$$
and
$$\begin{array}{r}
\tilde{W}(x,t)=\exp[-C_1g(y,Du(y,t_1))]\left\{1-\exp\left[-C_1(g(x,Du(x,t))-g(y,Du(y,t_1)))\right]\right.\\\left.\times\exp\left(C_1\frac{K}{2}\sum_{k=1}^{n-1}(u_k(x,t)-u_k(y,t_1))^2\right)\right\}-b|x-y|^2.
\end{array}$$
We consider the barrier function
$$v=-a_0|x-y|^2-h(d)+\psi(x')$$
with $h(d)=c_0(1-e^{-b_0d})$ and
$$\begin{array}{r}
\psi(x')=\exp(-C_1g(y,Du(y,t_1)))\left\{1-\exp\left(C_1K\sum_{k=1}^{n-1}({u_0}_k(x')-{u_0}_k(y))^2\right.\right.\\\left.\left.+2C_1K(|{u_0}_n(x')|^2+1)|D\rho(x')|^2\right)\right\}.
\end{array}$$
For suitable constants $a_0,b_0,c_0$ under control, we have $v\leq \tilde{W}$ on the parabolic boundary of $\overline{\Omega}\cap\overline{B_r}\times[0,T_1),$ and $Lv> C_2(1+|Dv|)$ in $\overline{\Omega}\cap\overline{B_r}\times[0,T_1)$ (see above the estimate of the mixed derivatives). By the comparison principle and since $v(y)=\tilde{W}(y,t_1)=0,$ we obtain that $v_n(y)\leq\tilde{W}_n(y,t_1),$ which gives 
$$u_{nn}(y,t_1)\leq C_4,$$
where $C_4$ is a constant under control. We finally observe that the bound is independent of $T_1.$
\section{Existence of the entire flow reduced to obtaining local $C^1$ and $C^2$ estimates, and convergence}\label{section entire solutions parabolic}
We suppose here that $F,$ $\ul u,$ $\ol u$ and $u_0$ satisfy the hypotheses of Theorem \ref{entire parabolic problem}: $\ul u,\ol u:\R^n\rightarrow\R$ are the barriers constructed Section \ref{section barriers} ($\ol u$ is a smooth spacelike and strictly convex function with constant mean curvature $h$ and $\ul u$ is the supremum of spacelike functions with constant scalar curvature $H_2=k$) and $u_0:\R^n\rightarrow\R$ is some smooth spacelike and strictly convex function such that
\begin{equation}\label{condition initial data}
\ul u<u_0<\ol u\mbox{ and }1\leq H_2[u_0]\leq k.
\end{equation}
To construct the entire flow, we solve a family of parabolic Dirichlet problems with parabolic boundary data $u_0$ on a growing sequence of balls. We modify slightly the equations near the boundary of the balls such that the problems are compatible on the corner of the parabolic domains. This is done in such a way that the normal velocities are uniformly bounded in terms of $H_2[u_0].$ Moreover the functions $\ul u,\ol u$ are still barriers for the modified evolution equations. This theorem is analogous to Theorem B.4 in \cite{BaySch} for the logarithmic Gauss curvature flow.

\begin{theorem} \label{theorem parabolic Dirichlet problem modified}
Let $R>1$ and the ball of $\R^n,$ $B_R:=\{x\in\R^n:\ |x|<R\}.$ We choose a smooth function $\eta:\overline{B_R}\times\R\rightarrow [0,1]$ such that  $\eta=0$ on $(B_{R-1}\times\R)\cup\graph\ul u\cup\graph\ol u,$ and $\eta=1$ near $\graph {u_0}_{|\partial B_R}.$ We also choose a smooth function $\zeta:[0,+\infty)\rightarrow[0,1]$ such that $\zeta(t)=1$ near $t=0,$ $\zeta(t)=0$ for $t\geq 1,$ and $\zeta'\leq 0.$ We set
\begin{equation}\label{definition hat f}
\hat f(x,u,t):=\eta(x,u)\zeta\left(\frac{t}{\varepsilon}\right)\left[H_2[u_0]^{\frac{1}{2}}-1\right]+1.
\end{equation}
If $\varepsilon>0$ is chosen sufficiently small, the parabolic Dirichlet problem 
\begin{equation}\label{modified parabolic Dirichlet problem}
\left\{\begin{array}{rcl}
-\frac{\dot u}{\sqrt{1-|Du|^2}}+H_2[u]^{\frac{1}{2}}& = & \hat f(x,u,t)\mbox{ in }B_R\times (0,+\infty)\\
u(x,t)& = &u_0(x)\mbox{ on }\partial B_R\times[0,+\infty)\cup B_R\times\{0\},
\end{array}\right.
\end{equation}
has an admissible solution $u:\ol{B_R}\times [0,+\infty)\rightarrow\R$ such that the  normal velocity $\frac{\dot u}{\sqrt{1-|Du|^2}}$ is uniformly bounded in terms of $H_2[u_0].$ Moreover 
\begin{equation}\label{u between barriers2}
\ul u\leq u\leq \ol u
\end{equation}
for all time. 
\end{theorem}
\noindent Observe that $u$ solution of (\ref{modified parabolic Dirichlet problem}) is also solution of (\ref{entire parabolic equation}) on $B_{R-1}\times[0,+\infty).$ We also write problem (\ref{modified parabolic Dirichlet problem}) in the equivalent form:
\begin{equation}\label{modified parabolic Dirichlet problem X}
\left\{\begin{array}{rcl}
\dot X&=&(F-\hat f)N\\
X&=&X_0\mbox{ on }\partial \Sigma_0^R\times[0,+\infty)\cup\Sigma_0^R\times\{0\},
\end{array}\right.
\end{equation}
where $\hat f$ is given by (\ref{definition hat f}), $\Sigma_0^R=\graph_{\ol B_R} u_0$ and $X:\Sigma_0^R\times[0,+\infty)\rightarrow\R^{n,1}$ is the embedding function.
\begin{proof}
The existence of a solution follows from Theorem \ref{theorem parabolic Dirichlet problem}: $\hat f_t\leq 0$ holds, the compatibility conditions are fulfilled on the corner of the parabolic domain, and, observing that
\begin{equation}
\hat f_{|t=0}=\eta H_2[u_0]^{\frac{1}{2}}+(1-\eta)1\leq H_2[u_0]^{\frac{1}{2}},
\end{equation}
we get  $(H_2^{\frac{1}{2}}-\hat f)_{|t=0}\geq 0.$ 

We now prove (\ref{u between barriers2}). We consider  the operators $P_1,P_2$ defined in (\ref{definition P1 P2}). We recall that $\ul u$ is of the form $\ul u=\sup_{i\in I}u_i,$ where for all $i\in I$ the function $u_i$ is spacelike with $H_2[u_i]=k.$ Defining $u_i(x,t):=u_i(x),$ we have $P_2[u_i]=H_2[u_i]^{\frac{1}{2}}=k^{\frac{1}{2}}.$ Moreover, from (\ref{definition hat f})-(\ref{condition initial data}), we have $P_2[u]=\hat f\leq k^{\frac{1}{2}}.$ Thus, since $u\geq u_i$ on the parabolic boundary, we get by the maximum principle that $u\geq u_i$ during the evolution, and finally that $u\geq\ul u.$ Now, setting $\ol u(x,t):=\ol u(x),$ we get
\begin{equation}\label{P1 upper barrier}
P_1[\ol u]=c(n)h.
\end{equation}
Moreover, since $\hat f\geq 1$ and recalling (\ref{mac laurin}), we have
\begin{equation}\label{P1 u}
P_1[u]=\hat f+c(n)H_1[u]-H_2[u]^{\frac{1}{2}}\geq 1.
\end{equation}
Since $c(n)h\leq 1$ (see Proposition \ref{proposition intro barriers}), (\ref{P1 upper barrier}) and (\ref{P1 u}) imply that $P_1[u]\geq P_1[\ol u],$ and thus that $u\leq\ol u$ since this inequality also holds on the parabolic boundary. 

We now estimate the normal velocity.  We consider $X:\Sigma_0^R\times[0,+\infty)\rightarrow\R^{n,1}$ solution of (\ref{modified parabolic Dirichlet problem X}). We first suppose that $t\in [0,\epsilon].$ We note that $|\hat f_t|\leq \frac{C_0}{\epsilon}$ where $C_0$ is a constant controlled in terms of $H_2[u_0];$ recalling inequality (\ref{velocity bound}) in Proposition \ref{lemma velocity bound}, we get 
$$0\leq F-\hat f\leq e^{Kt}\sup_{t=0}(F-\hat f)+\frac{C_0}{\epsilon K}(e^{Kt}-1),$$
where $K$ is a bound of $|N^{\alpha}\hat f_{\alpha}|.$ Observe that $K$ may grow with $R$ but is independent of $\epsilon$ (see Propositions \ref{C1 estimate maximum principle} and \ref{C1 estimate boundary}).
Taking $\epsilon=\frac{1}{K},$ we get the estimate: for all $t\in[0,\epsilon],$
\begin{equation}\label{velocity estimate t smaller epsilon}
0\leq F-\hat f\leq C\left(\sup_{t=0}(F-\hat f),C_0\right),
\end{equation}
where the right-hand side term is a constant bounded in terms of $H_2[u_0].$ For $t\geq\epsilon$ the normal velocity fulfills
$$\dt (F-\hat f)-F^{ij}(F-\hat f)_{ij}=-(F-\hat f)F^{ij}h_{ik}h^k{j},$$
and the maximum principle implies that
\begin{equation}\label{velocity estimate t larger epsilon}
0\leq F-\hat f\leq \sup_{t=\epsilon}(F-\hat f).
\end{equation}
Estimates (\ref{velocity estimate t smaller epsilon}) and (\ref{velocity estimate t larger epsilon}) imply the estimate of the normal velocity during the evolution in terms of the curvature $H_2[u_0]$ only.
\end{proof}
We set $u_R$ for a solution of the parabolic Dirichlet problem (\ref{modified parabolic Dirichlet problem}) on $\ol B_R\times[0,+\infty),$ $R>1.$ By construction, the normal velocity of the solution $u_R$ is bounded by a constant which is independent of $R.$ Suppose that the following estimates hold: for any $R_0\geq 0,$ there exists $R_1=R_1(R_0)$ larger than $R_0,$ $\theta\in(0,1)$ and $C\geq 0$ such that, for every $R\geq R_1,$
\begin{equation}\label{a priori estimates entire parabolic}
\sup_{\ol B_{R_0}\times[0,+\infty)}|Du_R|\leq 1-\theta,\ \sup_{\ol B_{R_0}\times[0,+\infty)}|u_R|+|D^2u_R|\leq C.
\end{equation}
Then, higher order derivative estimates (due to Krylov, Safonov, and Schauder for positive times) imply that a subsequence converges in
$$C^{\infty}(\R^n\times(0,+\infty))\cap C^{1,\alpha;0,\frac{\alpha}{2}}(\R^n\times[0,+\infty))$$
for every $0<\alpha<1$ to a solution $u\in C^{\infty}(\R^n\times(0,+\infty))\cap C^{1,1;0,1}(\R^n\times[0,+\infty))$ of (\ref{entire parabolic equation}).

We now prove the convergence of the entire flow $u$, following \cite{Ge1,E}. We first note that $H_2[u]^{\frac{1}{2}}\geq 1$ during the evolution, since this property holds for the solutions of the Dirichlet problems (\ref{modified parabolic Dirichlet problem}). We deduce that $\dot u\geq 0.$ Thus, for every $x\in\R^n,$ $t\mapsto u(x,t)$ is an increasing function, bounded above by $\ol u(x),$ and
$$u_{\infty}(x):=\lim_{t\rightarrow+\infty}u(x,t)$$
is well defined. The a priori estimates, locally uniform in space and uniform in time, show that $u_{\infty}$ is a smooth and spacelike function, and that $u(.,t)$ smoothly converges to $u_{\infty}$ as $t$ tends to $+\infty.$ We now fix $x\in\R^n.$ Since $\dot u\geq 0$ and $\sup_{t\in[0,+\infty)}|Du(x,t)|<1,$ we get from (\ref{entire parabolic equation}) that 
$$0\leq H_2[u]^{\frac{1}{2}}(x,t)-1\leq c(x)\dot u(x,t),$$
where $c(x)\in\R$ does not depend on $t.$ Thus
$$\int_0^{+\infty}[H_2[u]^{\frac{1}{2}}(x,t)-1]dt<+\infty,$$
and there exists a sequence $(t_n)$ such that $t_n\rightarrow+\infty$ and $H_2[u](x,t_n)\rightarrow 1$ as $n$ tends to infinity. Thus $u_{\infty}$ is a solution of (\ref{entire elliptic equation}).

In the rest of the article, we prove the local estimates (\ref{a priori estimates entire parabolic}).

\section{The local $C^1$ estimate}\label{section local C1 estimate}
We obtain here the local $C^1$ estimate in the spirit of R.Bartnik's estimate concerning hypersurfaces of prescribed mean curvature. See e.g. \cite{Btn}.
\subsection{The estimate with a general time function}$ $
We suppose that $X:\Sigma_0\times[0,+\infty)\rightarrow\R$ is a solution of (\ref{parabolic Dirichlet problem X}) with $\hat f\equiv 1.$ Let $\tau:\R^{n,1}\rightarrow \R$ be a smooth function whose lorentzian gradient $D\tau$ is a timelike and past oriented field ($\tau$ is a \textit{time function} on $\R^{n,1}$). We set $\alpha^{-2}=-\langle D\tau,D\tau\rangle$ and $T=-\alpha D\tau.$ Using the timelike vector field $T,$ we define a Riemannian metric on $\R^{n,1}$ by
$$|Y|_T^2:=\langle Y,Y\rangle+2\langle Y,T\rangle^2.$$
If $S$ is a tensor field on $\R^{n,1},$ we denote by $|S|_T$ its norm with respect to the metric $|.|_T.$ For instance, if $Y$ is a vector field on $\R^{n,1},$
$$|DY|_T=\left(\sum_{i,j=1}^{n+1}\langle D_{u_i}Y,u_j\rangle^2\right)^{\frac{1}{2}}$$
where $(u_1,\ldots,u_n,u_{n+1})$ is an orthonormal basis of $\R^{n,1}$ such that $u_{n+1}=T.$ We fix $\tau_0>0,$ and we suppose that, for every $t\geq 0,$ 
\begin{equation}\label{definition Sigma Tau0}
\Sigma_{t,\tau\geq\tau_0}:=\left\{X(x,t):\ x\in\Sigma_0,\ \tau(X(x,t))\geq\tau_0\right\}
\end{equation}
is a compact (spacelike) hypersurface such that $\tau=\tau_0$ on its boundary, and that the set $\Sigma_{t,\tau\geq 2\tau_0}$ is non-empty. We moreover suppose that $\tau,\alpha,\alpha^{-1},$ $|D\alpha|_T,$ $|DT|_T$ and $|D^2T|_T$ are bounded on the region 
\begin{equation}\label{def Dtau0}
D_{\tau_0}:=\bigcup_{t\geq 0}\Sigma_{t,\tau\geq\tau_0}
\end{equation}
of $\R^{n,1},$ and we fix $\tau_1$ such that $\tau_1\geq\tau$ on $D_{\tau_0}.$
Finally, we assume that $F$ has positive lower and upper bounds on $D_{\tau_0}.$ Setting here $\nu:=-\langle T,N\rangle,$ we prove the following estimate:
\begin{theorem} \label{theorem C1 interior estimate}
There exists a constant $C$ such that, for every $t\geq 0,$ 
\begin{equation}\label{local estimate nu}
\sup_{\Sigma_{t,\tau\geq 2\tau_0}}\nu\leq C.
\end{equation}
The constant $C$ depends on $\tau_0,\tau_1,$ upper bounds of  $\alpha,\alpha^{-1},|D\alpha|_T,$ $|DT|_T$ and $|D^2T|_T$ on the set $D_{\tau_0},$ on lower and upper bounds of $F$ on $D_{\tau_0},$ and on $\sup_{\Sigma_0,\tau\geq\tau_0}\nu.$
\end{theorem}
This estimate relies on the maximum principle applied to the function $\varphi=\eta\nu,$ where $\eta=(\tau-\tau_0)^K$ with $K$ a large constant. We fix $T_1\in(0,+\infty),$ and we consider $\varphi$ on the compact set
$$J=\{(x,t)\in\Sigma_0\times[0,T_1]:\ \tau(X(x,t))\geq\tau_0\}.$$
The function $\varphi$ reaches its maximum at a point $(x_0,t_0)\in J.$ Our purpose is to estimate $\varphi(x_0,t_0)$ by a constant under control which does not depend on $T_1;$ such an estimate clearly implies (\ref{local estimate nu}). If $t_0=0,$ we readily get that $\varphi(x_0,t_0)\leq C'$ where $C'=C'(K,\tau_0,\tau_1,\Sigma_{0,\tau\geq \tau_0})$ is a constant, and the result follows. We thus suppose that $t_0>0,$ and we prove the following proposition, which implies the estimate of $\varphi(x_0,t_0),$ and the theorem:
\begin{proposition} Let $\nu_0>1.$   If $K$ is sufficiently large, 
$$\nu(x_0,t_0)\leq\nu_0.$$
The constant $K$ depends on $\nu_0,$  $\tau_0,\tau_1,$ and on bounds on  $\alpha,\alpha^{-1},|DT|_T,$ $|D^2T|_T$ and $F$ on the region $D_{\tau_0}.$
\end{proposition}
In the following, we suppose that $\nu(x_0,t_0)>\nu_0,$ and we will obtain a contradiction if $K$ is large under control. Let us denote by $T^{||}$ the orthogonal projection of $T$ on the tangent space of $\Sigma_{t_0}.$ By definition of $T,$ we have $T^{||}=-\alpha\nabla\tau,$ and
\begin{equation}\label{norm gradient tau}
\nu^2-1=|T^{||}|^2=\alpha^2|\nabla\tau|^2.
\end{equation}
Let  $Z$ be the tangent vector field of $\Sigma_{t_0}$ such that $\langle Z,h\rangle=\langle N,D_hT\rangle,$ for all tangent vector $h.$ By a direct computation, 
\begin{equation}\label{gradient nu function S}
\nabla \nu=-S(T^{||})-Z=\alpha S(\nabla\tau)-Z.
\end{equation}
Here $S$ stands for the curvature endomorphism  of $\Sigma_{t_0}.$ We first estimate $|Z|$: 
\begin{lemma} \label{lemma norm Z} We have
\begin{equation}\label{norm Z}
|Z|\leq C_0\nu^2,
\end{equation}
where $C_0$ only depends on an upper bound of $|DT|_T.$
\end{lemma}
\begin{proof}
Let $u_1,\ldots,u_n,u_{n+1}$ be an orthonormal basis such that $u_{n+1}=T.$
Writing $N=\sum_{j=1}^n\alpha_ju_j+\nu T,$ we have $\sum_{j\leq n}\alpha_j^2=\nu^2-1.$ Let $h:=\sum_{i=1}^{n+1}h_iu_i$ be a vector tangent to $\Sigma_{t_0}.$ Since $\langle T,D_hT\rangle=0,$ 
\begin{equation}\label{estimate N DT}
|\langle N,D_hT\rangle|\leq\sum_{i=1}^{n+1}\sum_{j=1}^{n}|h_i\alpha_j\langle u_j,D_{u_i}T\rangle|\leq\nu|h|_T|DT|_T.
\end{equation}
Observe that inequality
\begin{equation}\label{comparison norms}
|h|_T\leq\sqrt{2}\nu|h|
\end{equation}
holds for all tangent vector $h.$ To prove this, suppose that $h\neq 0,$ and define $\displaystyle{\delta:=\frac{h_{n+1}^2}{\sum_{i=1}^nh_i^2}}.$ By a straightforward computation,
\begin{equation}\label{relation norms alpha}
|h|_T^2=\left(\frac{1+\delta}{1-\delta}\right)|h|^2.
\end{equation}
Since $h$ is a tangent vector, $\langle h,N\rangle=\sum_{i=1}^nh_i\alpha_i-h_{n+1}\nu=0,$ and the Schwarz inequality readily gives that $\delta\leq\frac{\nu^2-1}{\nu^2};$ inequality (\ref{comparison norms}) then follows from (\ref{relation norms alpha}). Finally, inequalities (\ref{estimate N DT}) and (\ref{comparison norms}) imply that $|\langle Z,h\rangle|\leq\sqrt{2}\nu^2|h||DT|_T,$ and  the lemma.
\end{proof}
\noindent We need the evolution equation of $\nu:$
\begin{lemma}\label{lemma new evolution equation nu}
During the evolution, 
\begin{eqnarray}
\dot\nu-F^{ij}\nu_{ij}&=&\langle D_NT,N\rangle+F^{ij}\langle D^2_{e_i,e_j}T,N\rangle+2F^{ij}\langle D_{e_i}T,D_{e_j}N\rangle\nonumber\\
&&-\nu F^{ij}h_j^kh_{ki}.\label{new evolution equation nu}
\end{eqnarray}
Here $(e_i)$ is a basis of $T\Sigma_{t}$ at $X(x,t).$ 
\end{lemma}
\begin{proof}
By Lemma 3.3 in \cite{Ge1}, the normal evolves according to $\dot N=\nabla F$ (recall that $\hat f\equiv 1$ here) and thus
 $$\langle T,\dot N\rangle=g^{ij}F_{;i}t_j=g^{ij}F^{kl}h_{kl;i}t_j.$$ 
Thus 
\begin{eqnarray}
\dot \nu&=&-\langle \dot T,N\rangle-\langle T,\dot N\rangle\nonumber
\\&=&-(F-\hat f)\langle D_NT,N\rangle-g^{ij}F^{kl}h_{kl;i}t_j.\label{dot nu}
\end{eqnarray}
We fix $(x,t)\in \Sigma_0\times[0,+\infty).$ Let $(e_i)$ be a local frame of $\Sigma_{t}$ such that $\nabla_{e_i}e_j(X(x,t))=0.$ Thus $D_{e_i}e_j(X(x,t))=h_{ij}N,$ and we have 
$$D_{e_j}(D_{e_i}T)=D^2_{e_i,e_j}T+h_{ij}D_NT$$
and
$$D_{e_j}(D_{e_i}N)=D_{e_j}(h_i^ke_k)=h^k_{i;j}e_k+h_i^kh_{kj}N.$$
Thus 
\begin{eqnarray}
F^{ij}\nu_{ij}&=&-F^{ij}\langle D^2_{e_i,e_j}T,N\rangle-F\langle D_NT,N\rangle-2F^{ij}\langle D_{e_i}T,D_{e_j}N\rangle\nonumber
\\&&-F^{ij}h^k_{i;j}\langle T,e_k\rangle+F^{ij}h_i^kh_{kj}\nu.\label{nu ij}
\end{eqnarray}
Equations (\ref{dot nu}) and (\ref{nu ij}), together with the symmetry of $h_{ij;k}$ in the three indices (the Codazzi equations), imply (\ref{new evolution equation nu}).
\end{proof}
Thus, the maximum condition $\dt (\log\varphi)-F^{ij}(\log\varphi)_{ij}\geq 0$ reads
\begin{equation}\label{maximum condition C1 interior estimate}
\begin{array}{l}
\left(F^{ij}h_j^kh_{ki}-\frac{1}{\nu^2}F^{ij}\nu_i\nu_j\right)-\left(\dt\log\eta-F^{ij}\left(\log\eta\right)_{ij}\right)
\\\leq\frac{1}{\nu}\left(\langle D_NT,N\rangle+F^{ij}\langle D^2_{e_i,e_j}T,N\rangle+2F^{ij}\langle D_{e_i}T,D_{e_j}N\rangle\right).
\end{array}
\end{equation}
\noindent Moreover, since $\nu_i=-h_i^k t_{k}-Z_i$ (see (\ref{gradient nu function S})), we get:
\begin{eqnarray}
F^{ij}h_j^kh_{ki}-\frac{1}{\nu^2}F^{ij}\nu_i\nu_j&=&\left(F^{ij}h_j^kh_{ki}-\frac{1}{\nu^2}F^{ij}h_i^kt_kh_j^lt_l\right)\nonumber\\&&-\frac{1}{\nu^2}F^{ij}Z_iZ_j-\frac{2}{\nu^2}F^{ij}h_i^kt_kZ_j.
\end{eqnarray}
We now estimate the terms in the expressions above:
\begin{lemma} \label{estimates various quantities}
There exist constants $C_1,C_2,C_3$ such that
\begin{equation}\label{estimate1}
\left|\langle D_NT,N\rangle\right|\leq C_1\nu^2,
\end{equation}
\begin{equation}\label{estimate2}
\left|F^{ij}\langle D^2_{e_i,e_j}T,N\rangle+\frac{1}{\nu}F^{ij}Z_iZ_j\right|\leq C_2\sigma_1\nu^3,
\end{equation}
\begin{equation}\label{estimate3}
\mbox{ and }\left|\frac{2}{\nu}F^{ij}h_i^kt_kZ_j+2F^{ij}\langle D_{e_i}T,D_{e_j}N\rangle\right|\leq C_3\nu^2\sum_{i}\sigma_{1,i}|\lambda_i|.
\end{equation}
The constants $C_1,C_2,C_3$ only depend on bounds on $|DT|_T,$ $|D^2T|_T$ and $F.$ 
\end{lemma}
\begin{proof}
Let $(u_i)_{1\leq i\leq n+1}$ be an orthonormal basis of $\R^{n,1}$ such that $u_{n+1}=T.$ We suppose that $e_1,\ldots,e_n$ is an orthonormal basis of principal directions of $\Sigma_{t_0}$ at $x_0,$ and we define $(\alpha_{ij})_{1\leq i,j\leq n+1}$ such that 
$$e_i=\sum_{j=1}^{n+1}\alpha_{ij}u_j\mbox{ and } N=\sum_{j=1}^{n+1}\alpha_{n+1,j}u_j.$$
The following estimates hold: for all $i=1,\ldots,n+1,$ $\displaystyle{\sum_{j=1}^{n+1}\alpha_{ij}^2\leq 2\nu^2}$ (see (\ref{comparison norms}) for $i<n+1,$ and use $\langle N,N\rangle=-1$ for $i=n+1).$ 

We first prove (\ref{estimate1}): using the Schwarz inequality, 
\begin{eqnarray*}
\left|\langle N,D_NT\rangle\right|&\leq&\sum_{i,j=1}^{n+1}|\alpha_{n+1,i}||\alpha_{n+1,j}||\langle u_i,D_{u_j}T\rangle|\leq 2\nu^2|DT|_{T},
\end{eqnarray*}
and we get (\ref{estimate1}). 

We estimate the first term in (\ref{estimate2}):
\begin{eqnarray*}
\left|\sum_{ij}F^{ij}\langle N,D^2_{e_i,e_j}T\rangle\right|&\leq&\frac{1}{2F}\sum_i\sigma_{1,i}\sum_{jkl}|\alpha_{n+1,j}||\alpha_{ik}||\alpha_{il}|\left|\langle u_j,D^2_{u_k,u_l}T\rangle\right|\\
&\leq&\frac{1}{2F}(n-1)2^{3/2}\sigma_1\nu^3|D^2T|_T.
\end{eqnarray*}
For the second term in (\ref{estimate2}), we readily get: $\displaystyle{F^{ij}Z_iZ_j\leq\frac{n-1}{2F}\sigma_1|Z|^2,}$ and we obtain the estimate from (\ref{norm Z}). 

We finally prove (\ref{estimate3}):
$$\left|\frac{2}{\nu}F^{ij}h_i^kt_kZ_j+2F^{ij}\langle D_{e_i}T,D_{e_j}N\rangle\right|\leq \frac{1}{F}\sum_i\sigma_{1,i}|\lambda_i|\left(\frac{1}{\nu}|t_i||Z_i|+|\langle D_{e_i}T,e_i\rangle |\right);$$
using $|T^{||}|\leq\nu,$ $|Z|\leq C_0\nu^2$ (Lemma \ref{lemma norm Z}) and $|\langle D_{e_i}T,e_i\rangle |\leq 2\nu^2|DT|_T$ we obtain (\ref{estimate3}).
\end{proof}
We now prove that the smallest principal curvature of $\Sigma_{t_0}$ at $x_0$ is negative, with absolute value arbitrarily large, if the constant $K$ is chosen sufficiently large:
\begin{lemma}\label{lower bound lambda_n}
If 
\begin{equation}\label{inequality K}
K\geq 2C_0\alpha(\tau_1-\tau_0)\frac{\nu_0^2}{\nu_0^2-1},
\end{equation}
the smallest principal curvature $\lambda_n$ is negative and is estimated by
\begin{equation}\label{estimate lambda n}
\lambda_n\leq -\frac{1}{2}\frac{\alpha^{-1}K}{\tau-\tau_0}\nu.
\end{equation}
\end{lemma}
Note that since $\varphi$ vanishes if $\tau=\tau_0$ and is positive if $\tau>\tau_0,$ we have $\tau>\tau_0$ at $(x_0,t_0).$
\begin{proof}
By (\ref{gradient nu function S}),
$$S(\nabla\tau)=\frac{1}{\alpha}(\nabla\nu+Z);$$
moreover, the extremum condition $\nabla\varphi(x_0,t_0)=0$ reads
$$\frac{\nabla\nu}{\nu}=-\frac{K}{\tau-\tau_0}\nabla\tau.$$
Thus
\begin{equation}\label{expression S nabla tau} 
\left\langle S\left(\frac{\nabla\tau}{|\nabla\tau|}\right),\frac{\nabla\tau}{|\nabla\tau|}\right\rangle=-\nu\alpha^{-1} \frac{K}{\tau-\tau_0}+\frac{\alpha^{-1}}{|\nabla\tau|^2}\langle Z,\nabla\tau\rangle.
\end{equation}
By Lemma \ref{lemma norm Z} and expression (\ref{norm gradient tau}), we estimate
$$\frac{\alpha^{-1}}{|\nabla\tau|^2}\left|\langle Z,\nabla\tau\rangle\right|\leq\alpha^{-1}\frac{|Z|}{|\nabla\tau|}\leq C_0\frac{\nu^3}{\nu^2-1}.$$
Since $\nu>\nu_0>1,$ we have $\frac{\nu^2}{\nu^2-1}\leq\frac{\nu_0^2}{\nu_0^2-1},$ and thus
\begin{eqnarray*}
-\nu\alpha^{-1} \frac{K}{\tau-\tau_0}+ C_0\frac{\nu^3}{\nu^2-1}&\leq&\nu\left(-\alpha^{-1} \frac{K}{\tau-\tau_0}+C_0\frac{\nu_0^2}{\nu_0^2-1}\right)\\
&\leq&-\frac{1}{2}\alpha^{-1} \frac{K}{\tau-\tau_0}\nu
\end{eqnarray*}
if $\frac{1}{2}\alpha^{-1} \frac{K}{\tau-\tau_0}\geq C_0\frac{\nu_0^2}{\nu_0^2-1},$ which follows from the hypothesis (\ref{inequality K}). Thus expression (\ref{expression S nabla tau}) is bounded by $\displaystyle{-\frac{1}{2}\alpha^{-1} \frac{K}{\tau-\tau_0}\nu},$ and the result follows.
\end{proof}
As in Lemma 4.6. of \cite{Bay2}, we obtain the key inequality:
\begin{lemma}\label{key inequality epsilon}
If $\lambda_n\leq 0,$
$$F^{ij}h_j^kh_{ki}-\frac{1}{\nu^2}F^{ij}h_i^kt_kh_j^lt_l\geq\varepsilon\sigma_{1,n}\lambda_n^2,$$
where $\varepsilon$ is a positive constant which only depends on an upper bound of $F.$
\end{lemma}
\begin{proof}
Just follow the lines of the proof of Lemma 4.6. in \cite{Bay2}, with $\lambda_n$ instead of $\lambda_{i_0}$ and $t_i$ instead of $u_i,$ using here identity (\ref{norm gradient tau}).
\end{proof}
Inequality (\ref{maximum condition C1 interior estimate}) and Lemmas \ref{estimates various quantities} and \ref{key inequality epsilon} thus imply that
\begin{equation}\label{inequality with sigma 1,i lambda i}
\varepsilon\sigma_{1,n}\lambda_n^2-\left(\dt\log\eta-F^{ij}\left(\log\eta\right)_{ij}\right)\leq C_1\nu+C_2\sigma_1\nu^2+C_3\nu\sum_i\sigma_{1,i}|\lambda_i|.
\end{equation}
The next lemma permits to balance the last term in (\ref{inequality with sigma 1,i lambda i}):
\begin{lemma} \label{estimate sigma 1,i lambda i}
If $\lambda_n\leq 0,$ then, for all $i,$
$$\sigma_{1,i}|\lambda_i|\leq\sigma_2+\frac{\beta}{\nu}\sigma_{1,n}\lambda_n^2,$$
where $\beta$ is a constant arbitrarily small, if $K=K(\alpha,\beta,\tau_0,\tau_1)$ is chosen sufficiently large.
\end{lemma}
\begin{proof}
We first suppose that $\lambda_i\leq 0:$ we thus have $|\lambda_i|\leq|\lambda_n|,$ and, since $\sigma_{1,n}\geq\sigma_{1,i},$
$$\sigma_{1,i}|\lambda_i|\leq\sigma_{1,n}|\lambda_n|.$$ 
By estimate (\ref{estimate lambda n}) in Lemma \ref{lower bound lambda_n}, $\frac{\nu}{\beta}\leq|\lambda_n|,$ if $K=K(\tau_0,\tau_1,\alpha,\nu_0,\beta)$ is sufficiently large. This implies
$$\sigma_{1,i}|\lambda_i|\leq\frac{\beta}{\nu}\sigma_{1,n}\lambda_n^2,$$
and the estimate. We now suppose that $\lambda_i\geq 0:$ the proof relies on the following inequality:
\begin{equation}\label{inequality lambda_i sigma_1,i}
\lambda_i\sigma_{1,i}\leq\sigma_2+\gamma_n\sigma_{1,n}|\lambda_n|,
\end{equation}
where $\gamma_n$ is a positive constant which only depends on $n.$ With this inequality at hand, taking as above $K$ sufficiently large such that $\frac{\gamma_n\nu}{\beta}\leq|\lambda_n|,$ we obtain the lemma. To finish the proof we thus focus on the proof of (\ref{inequality lambda_i sigma_1,i}):
first, 
$$\lambda_i\sigma_{1,i}=\sigma_2-\sigma_{2,in}-\sigma_{1,ni}\lambda_n,$$
where $-\sigma_{2,in}$ is a sum of terms of the form $-\lambda_j\lambda_k,$ with $j\neq k,$ and $j,k\notin\{i,n\}.$ We observe that $-\sigma_{1,ni}\lambda_n=\sigma_{1,ni}|\lambda_n|\leq\sigma_{1,n}|\lambda_n|,$ since $\lambda_i\geq 0.$ It thus remains to bound the positive terms appearing in $-\sigma_{2,in}:$ let $j,k\notin\{i,n\}$ be such that $-\lambda_j\lambda_k>0.$ We suppose that $\lambda_k<0,$ and thus that $\lambda_j>0.$ We have $\lambda_j\leq\sigma_{1,n}$ since $\sigma_{1,n}=\lambda_j+\sigma_{1,j}-\lambda_n$ with $\sigma_{1,j}$ and $-\lambda_n$ positive. Moreover $|\lambda_k|\leq|\lambda_n|$ since $\lambda_n\leq\lambda_k\leq 0.$ Thus $-\lambda_j\lambda_k\leq\sigma_{1,n}|\lambda_n|,$ which concludes the proof of (\ref{inequality lambda_i sigma_1,i}). 
\end{proof}
From (\ref{inequality with sigma 1,i lambda i}) and Lemma \ref{estimate sigma 1,i lambda i} we obtain the inequality
\begin{equation}
\varepsilon\sigma_{1,n}\lambda_n^2-\left(\dt\log\eta-F^{ij}\left(\log\eta\right)_{ij}\right)\leq C_1\nu+C_2\sigma_1\nu^2+nC_3\nu\left(\sigma_2+\frac{\beta}{\nu}\sigma_{1,n}\lambda_n^2\right).
\end{equation}
Choosing $\beta$ small such that $\varepsilon-nC_3\beta\geq\frac{\varepsilon}{2},$ we obtain 
\begin{equation}
\frac{\varepsilon}{2}\sigma_{1,n}\lambda_n^2-\left(\dt\log\eta-F^{ij}\left(\log\eta\right)_{ij}\right)\leq C_1\nu+C_2\sigma_1\nu^2+nC_3\nu\sigma_2.
\end{equation}
We now estimate the contribution of the cut-off function:
\begin{lemma} There exists a constant $C_4$ such that
$$\dt\left(\log\eta\right)-F^{ij}\left(\log\eta\right)_{ij}\leq C_4K\nu^2\frac{\sigma_1}{\tau-\tau_0}\left(1+\frac{1}{\tau-\tau_0}\right).$$
The constant $C_4$ depends on bounds on $\alpha^{-1},|D\alpha|_T,|DT|_T$ and $F.$
\end{lemma}
\begin{proof}
By a direct computation, 
$$\dt(\log\eta)=(F-\hat f)\frac{K}{\tau-\tau_0}D\tau(N).$$
Since $|D\tau(N)|=\nu\alpha^{-1},$ we get
\begin{equation}\label{estimate dt log eta}
\dt(\log\eta)\leq C\frac{K}{\tau-\tau_0}\nu,
\end{equation}
where $C$ depends on  bounds on $\alpha^{-1}$ and $F-\hat f.$ Moreover
\begin{equation}\label{log eta ii}
\sum_i\sigma_{1,i}(\log\eta)_{ii}=K\left(\sum_i\sigma_{1,i}\frac{\tau_{ii}}{\tau-\tau_0}-\sum_i\sigma_{1,i}\frac{\tau_i^2}{(\tau-\tau_0)^2}\right).
\end{equation}
Using (\ref{norm gradient tau}), the last term is directly estimated by 
\begin{equation}\label{estimate log eta ii 1}
 \sum_i\sigma_{1,i}\frac{\tau_i^2}{(\tau-\tau_0)^2}\leq(n-1)\alpha^{-2}\nu^2\frac{\sigma_1}{(\tau-\tau_0)^2}.
 \end{equation}
To estimate the first term in (\ref{log eta ii}), we extend  $e_1,\ldots,e_n$ to vector fields of $\Sigma_{t_0}$ such that $\nabla_{e_i}e_i(x_0)=0.$ Since $\tau_i=-\langle\alpha^{-1}T,e_i\rangle$ and $D_{e_i}e_i=\lambda_iN,$ we get 
$$\tau_{ii}=-\langle D_{e_i}(\alpha^{-1}T),e_i\rangle+\alpha^{-1}\nu\lambda_i.$$
Since $\langle D_{e_i}(\alpha^{-1}T),e_i\rangle=(\alpha^{-1})_i\langle T,e_i\rangle+\alpha^{-1}\langle D_{e_i}T,e_i\rangle,$ and since
$$|\langle T,e_i\rangle|\leq \nu,\ |\langle D_{e_i}T,e_i\rangle|\leq 2\nu^2|DT|_T, \mbox{ and }|(\alpha^{-1})_i|\leq\sqrt{2}\nu|D\alpha^{-1}|_T$$
(see the proof of Lemma \ref{estimates various quantities}), we obtain
\begin{equation*}\label{estimate log eta ii 2}
\left|\sum_i\sigma_{1,i}\frac{\tau_{ii}}{\tau-\tau_0}\right|\leq \frac{\alpha^{-1}}{\tau-\tau_0}\left(2(n-1)\sigma_1\nu^2\left\{\frac{1}{\alpha^{-1}}|D\alpha^{-1}|_T+|DT|_T\right\}+2\nu\sigma_2\right).
\end{equation*}
Observing moreover that $\sigma_1$ is bounded below by a positive constant (Mac-Laurin inequality), this last estimate together with (\ref{log eta ii}) and (\ref{estimate log eta ii 1}) imply that
\begin{equation}\label{estimate log eta ii 3}
-F^{ij}\left(\log\eta\right)_{ij}\leq CK\nu^2\frac{\sigma_1}{\tau-\tau_0}\left(1+\frac{1}{\tau-\tau_0}\right)
\end{equation}
where $C$ only depends on bounds on $\alpha^{-1},|D\alpha|_T,|DT|_T,$ and $F.$ Finally, estimates (\ref{estimate dt log eta}) and (\ref{estimate log eta ii 3}) give the result.
\end{proof}

Using the lower bound $\displaystyle{\frac{\varepsilon}{2}\sigma_{1,n}\lambda_n^2\geq\frac{\varepsilon}{2}\sigma_1\frac{\alpha^{-2}K^2}{4(\tau-\tau_0)^2}\nu^2}$ given by Lemma \ref{lower bound lambda_n}, we finally obtain
$$\frac{\varepsilon}{2}\sigma_1\frac{1}{4}K^2\nu^2-C_4K\nu^2\alpha^2\sigma_1((\tau-\tau_0)+1)\leq \left(C_1\nu+C_2\sigma_1\nu^2+nC_3\nu\sigma_2\right)\alpha^2(\tau-\tau_0)^2,$$
which is impossible if $K$ is sufficiently large under control.
\subsection{The construction of the adapted time function}
We suppose that $\ul u$ and $\ol u$ are the barriers constructed Section \ref{section barriers}. Let $R_0> 0,$ and consider $\Psi:\R^n\rightarrow\R$ such that $\Psi<\ul u$ on $\ol B_{R_0},$ and $\Psi\geq\ol u$ near infinity, given by Lemma \ref{construction psi}. We define
$$\tau(x_1,\ldots,x_n,x_{n+1}):=x_{n+1}-\Psi(x_1,\ldots,x_n).$$
The function $\tau$ is a time function since $\psi$ is spacelike. We fix $\tau_0>0$ such that $\Psi+2\tau_0\leq\ul u\mbox{ on }\ol B_{R_0}.$ We also fix $R_1$ such that $\psi(x)\geq\ol u(x)$ if $|x|\geq R_1,$ and, for $R\geq R_1+1,$ we consider $X^R:\Sigma_0^R\times[0,+\infty)\rightarrow\R^{n,1}$ solution of (\ref{modified parabolic Dirichlet problem X}). We set $K$ for the compact set $\{\psi(x_1,\ldots,x_n)\leq x_{n+1}\leq\ol u(x_1,\ldots,x_n)\}.$ Note that $\hat f\equiv 1$ on $K,$ and that the normal velocity $F$ is bounded from above and from below during the evolution, uniformly in $R.$  For all $t\in[0,+\infty),$ the set $\Sigma_{t,\tau\geq\tau_0}^R$ defined by (\ref{definition Sigma Tau0}) is compact, and is such that $\tau=\tau_0$ on its boundary. Moreover, since the set $D_{\tau_0}$ defined by (\ref{def Dtau0}) belongs to the compact set $K,$ the time function $\tau$ satisfies all the requirements  of the previous section. Thus Theorem \ref{theorem C1 interior estimate} applies and gives the gradient estimate on the set
$$\{(x,t)\in\Sigma_0^R\times[0,+\infty):\ \tau(X^R(x,t))\geq 2\tau_0\}.$$
We deduce the required local gradient estimate (\ref{a priori estimates entire parabolic}) since
$$\graph_{\ol B_{R_0}}u_R\subset\{X\in\R^{n,1}:\ \tau(X)\geq 2\tau_0\}.$$
\section{The local $C^2$ estimate}\label{section local C2 estimate}
We suppose that $\ul u,\ol u $ are the barriers constructed Section \ref{section barriers}.
\begin{theorem}
Let $R_0\geq 0,$ and let ${R_0}'\geq R_0$ and $\Phi:\ol B_{{R_0}'}\rightarrow\R$ be such that
$$\Phi>\ol u\mbox{ on }\ol B_{R_0},\ \Phi\leq\ul u\mbox{ on }\partial B_{{R_0}'},$$
given by Lemma \ref{construction phi} . We fix $\delta_0>0$ such that $\Phi\geq\ol u+\delta_0$ on $\ol B_{R_0},$ and we set,  for all $X=(x_1,\ldots,x_{n+1})\in\ol B_{{R_0}'}\times\R\subset\R^{n,1},$ 
$$\eta(X):=\Phi(x_1,\ldots,x_n)-x_{n+1}.$$ 
Then, there exists $R_1\geq {R_0}'$ such that, for all $R\geq R_1,$ the solution $X^R$ of (\ref{modified parabolic Dirichlet problem X}) satisfies the following local $C^2$ estimate:
$$\sup_{\{(x,t):\ \eta(X^R(x,t))\geq\delta_0\}}|II^R_{(x,t)}|\leq C.$$
Here $|II^R_{(x,t)}|$ stands for the norm of the second fundamental form of $\Sigma^R_{t}$ at $X^R(x,t),$ and $C$ is a constant controlled by the local $C^1$ estimate on the set where $\eta\geq 0.$
\end{theorem}
\begin{proof}
We fix $R_1\geq {R_0}'$ such that for all $R\geq R_1$ the local $C^1$ estimate holds on the set where $\ul u\leq x_{n+1}\leq \Phi.$ For $T_1>0$ we define 
$$\Gamma_{T_1}=\{(x,t)\in\Sigma_0^R\times[0,T_1]:\ \eta(X^R(x,t))\geq 0\}.$$ 
Following J.Urbas \cite{U}, we suppose that the function
\begin{equation}
\tilde W(x,t,\xi):=\eta^{\beta}(X^R(x,t))h_{ij}\xi^i\xi^j,
\end{equation}
defined for all $(x,t)\in\Gamma_{T_1}$ and all unit $\xi\in T_{X^R(x,t)}\Sigma_t,$ reaches its maximum at a point $(x_0,t_0,\xi_0),$ with $t_0>0.$ Arguing as in Paragraph \ref{Dirichlet problem C2 max pple}, we also get here inequality (\ref{inequality fundamental C2 interior estimate reduced}) at $(x_0,t_0).$ Moreover, using the evolution equation (\ref{evolution height u}) of $u=x_{n+1},$ we obtain
\begin{equation}
\dt\log\eta-F^{ij}\left(\log\eta\right)_{ij}=\frac{\nu\hat f}{\eta}+\frac{1}{\eta}\left(\dot\Phi-F^{ij}\Phi_{ij}\right)+F^{ij}\frac{\eta_i\eta_j}{\eta^2}.
\end{equation} 
The following estimates hold:  $F^{ij}\Phi_{ij}\geq c_0\sum_iF^i_i-c,$ $\dot\Phi\leq c,$ and $\nu\hat f\leq c$ where $c_0$ and $c$ are controlled constants; for the first estimate we refer to \cite{U} p.313, and for the second estimate to (\ref{dot Phi}). Thus
\begin{equation}\label{estimate evolution eta}
-\beta\left(\dt\log\eta-F^{ij}\left(\log\eta\right)_{ij}\right)\geq\frac{\beta}{\eta}\left(c_0\sum_iF^i_i-c'\right)-\beta F^{ij}\frac{\eta_i\eta_j}{\eta^2}.
\end{equation}
Inequalities (\ref{inequality fundamental C2 interior estimate reduced})  and (\ref{estimate evolution eta}) give inequality (2.8) obtained by J.Urbas in \cite{U} p. 312 (where the first term in (2.8) is moreover estimated by (2.12) \cite{U} p. 313). We then follow the arguments in \cite{U}, and obtain an upper bound of $\tilde W(x_0,t_0,\xi_0),$ if $\beta$ is chosen sufficiently large (under control). The bound is independent of $T_1$ and $R.$ This gives an upper bound of the second fundamental form during the evolution, on the set where $\eta\geq\delta_0.$ 
\end{proof}
This estimate implies the local $C^2$ estimate (\ref{a priori estimates entire parabolic}) since
$$\graph_{\ol B_{R_0}} u_R\subset \{X\in\R^{n,1}:\ \eta(X)\geq\delta_0\},$$
and thus completes the proof of Theorem \ref{entire parabolic problem}.
\appendix
\section{}\label{comparison mean curvature section}
\begin{lemma}\label{comparison mean curvature}
Let $u$ be a spacelike and convex function defined on $\R^n,$ and let ${x'}_0\in\R^{n-2}.$ Setting
$$\tilde u(x_1,x_2):=u(x_1,x_2,{x'}_0),$$
we have, for all $(x_1,x_2)\in\R^2,$
$$H_1[\tilde u](x_1,x_2)\leq H_1[u](x_1,x_2,{x'}_0).$$
\end{lemma}
\begin{proof}
The second fundamental form of the graph of $\tilde u$ in the chart $(x_1,x_2)$ is
$$\tilde{II}=\frac{1}{\sqrt{1-|D\tilde u|^2}}D^2\tilde u.$$
Since $|Du|\geq |D\tilde u|$ and $D^2\tilde u=D^2u_{|_{\R^2\times\{0\}}},$ we obtain
$$\frac{1}{\sqrt{1-|D\tilde u|^2}}D^2\tilde u\leq \frac{1}{\sqrt{1-|Du|^2}}D^2u_{|_{\R^2\times\{0\}}}.$$
The right-hand side term is the second fundamental form $II$ of the graph of $u,$ in the chart $(x_1,\ldots,x_n),$ restricted to the plane $x_3=\cdots=x_n=0.$ Let us fix $(e_1,e_2,\ldots,e_n)$ a basis in the chart $(x_1,\ldots,x_n)$ which induces an orthonormal basis $(\hat e_1,\ldots,\hat e_n)$ of the tangent space of $\graph u$ at $(x_1,x_2,{x'}_0).$ We suppose moreover that $e_1,e_2$ belong to the plane $x_3=\cdots=x_n=0.$ Thus $\hat e_1,\hat e_2$ are tangent to $\graph\tilde u,$ and we get:
\begin{eqnarray*}
H_1[\tilde u](x_1,x_2)=\tilde{II}(\hat e_1)+\tilde{II}(\hat e_2)&\leq& II(\hat e_1)+II(\hat e_2)\\
&\leq&\sum_{i=1}^n II(\hat e_i)\leq H_1[u](x_1,x_2,{x'}_0),
\end{eqnarray*}
where the second inequality follows from the convexity of $u.$
\end{proof}


\begin{thebibliography}{99}
\bibitem[1] {Btn} R.Bartnik. Maximal surfaces and General Relativity. Proc. CMA. \textbf{12} (1987) 24-49.
\bibitem[2] {BS} R.Bartnik, L.Simon. Spacelike hypersurfaces with prescribed boundary values and mean curvature. Comm. Math. Phys. \textbf{87} (1982) 131-152.
\bibitem[3] {Bay} P.Bayard. Dirichlet problem for space-like hypersurfaces with prescribed scalar curvature in $\R^{n,1}$. Calc. Var. \textbf{18} (2003) 1-30.
\bibitem[4] {Bay2} P.Bayard. Entire space-like hypersurfaces with prescribed scalar curvature in $\R^{n,1}$. Calc. Var. \textbf{26:2} (2006) 245-264.
\bibitem[5] {BaySch} P.Bayard, O.Schn\"urer. Entire spacelike hypersurfaces of constant Gauss curvature in Minkowski space. J. reine angew. Math. To appear.  
\bibitem[6] {CT} H.Choi, A.Treibergs. Gauss maps of spacelike mean curvature hypersurfaces of Minkowski space. J. Diff. Geom. \textbf{32:3} (1990) 775-817.
\bibitem[7] {De} Ph.Delano\"e. Dirichlet problem for the equation of prescribed Lorentz-Gauss curvature. Ukr. Math. J. \textbf{42:12} (1990) 1538-1545.
\bibitem[8] {Ecker} K.Ecker. Interior estimates and longtime solutions for mean curvature flow on noncompact spacelike hypersurfaces in Minkowski space. J. Diff. Geom. \textbf{46} (1997) 481-498.
\bibitem[9] {E} C.Enz. The scalar curvature flow in Lorentzian manifolds. Preprint (2007).
\bibitem[10] {Ge1} C.Gerhardt. Hypersurfaces of prescribed curvature in Lorentzian manifolds. Indiana Univ. Math. J. \textbf{49} (2000) 1125-1153.
\bibitem[11] {Ge2} C.Gerhardt. Hypersurfaces of prescribed scalar curvature in Lorentzian manifolds. J. reine angew. Math. \textbf{554} (2003) 157-199.
\bibitem[12] {GJS} B.Guan, H.Jian, R.Schoen. Entire spacelike hypersurfaces of prescribed Gauss curvature in Minkowski space. J. reine angew. Math. \textbf{595} (2006) 167-188.
\bibitem[13] {HN} J.Hano, K.Nomizu. On isometric immersions of hyperbolic plane into the Lorentz-Minkowski space and the Monge-Amp\`ere equation of a certain type. Math. Ann. \textbf{262:1} (1983) 245-253.
\bibitem[14] {IL1} N.M.Ivochkina, O.Ladyzhenskaya. Estimation of the second derivatives for surfaces evolving under the action of their principal curvatures. Topological Methods in nonlinear Analysis. \textbf{6} (1995) 265-282.
\bibitem[15] {IL2} N.M.Ivochkina, O.Ladyzhenskaya. Estimation of the second derivatives on the boundary for surfaces evolving under the influence of their principal curvatures. St. Petersburg Math. J. \textbf{9:2} (1998) 199-217.
\bibitem[16] {ILT} N.M.Ivochkina, M.Lin, N.S.Trudinger. The Dirichlet problem for the prescribed curvature quotient equations with general boundary values. Jost J. (ed.) Geometric analysis and the calculus of variations, Cambridge (1996) 125-141.
\bibitem[17] {T} A.Treibergs. Entire Spacelike hypersurfaces of constant mean curvature in Minkowski space. Seminar on Differential Geometry. Ann. of Math. Stud. \textbf{102} (1982) 229-238.
\bibitem[18] {T2} N.S.Trudinger. On the Dirichlet problem for Hessian equations. Acta Math. \textbf{175} (1995) 151-164.
\bibitem[19] {U} J.Urbas. The Dirichlet problem for the equation of prescribed scalar curvature in Minkowski space. Calc. Var. \textbf{18} (2003) 307-316.
\end{thebibliography}
\end{document}